\documentclass[reqno,12pt] {amsart}
\usepackage[width=6.5in, left=0.95in, right=1.05in, height=8.8in, top=1.1in,bottom=1.1in]{geometry}
\usepackage{amsfonts}
\usepackage[active]{srcltx}
\usepackage{amsmath,amssymb,amsthm}
\usepackage{mathtools}
\usepackage{bm} 
\usepackage[shortlabels] {enumitem}
\usepackage{wasysym}
\usepackage{verbatim}
\usepackage{pdfsync}
\usepackage[]{microtype}
\usepackage{graphicx}
\usepackage[
bookmarks=true,         
bookmarksnumbered=true, 
colorlinks=true, pdfstartview=FitV, linkcolor=blue, citecolor=blue,
urlcolor=blue]{hyperref}
\usepackage[backrefs,abbrev,msc-links]{amsrefs}

\theoremstyle{plain}
\newtheorem{theorem}{Theorem}[section]

\newtheorem{lemma}[theorem]{Lemma}
\newtheorem{corollary}[theorem]{Corollary}

\newtheorem{proposition}[theorem]{Proposition}

\theoremstyle{remark}

\theoremstyle{definition}
\newtheorem{remark}[theorem]{Remark}

\def\RR{\mathbb{R}}

\def\EE{\mathbb{E}}
\def\PP{\mathbb{P}}
\def\NN{\mathbb{N}}

\def\cff{\mathcal{F}}

\def\cgg{\mathcal{G}}
\def\cii{\mathcal{I}}

\def\lt{\left}
\def\rt{\right}

\let\Section=\section
\def\section{\setcounter{equation}{0}\Section}

\title[Branching Gaussian Processes] {Laws of Large Numbers for Supercritical Branching Gaussian Processes}

\author[M. Kouritzin] {Michael A. Kouritzin}
\thanks{Research supported in part by NSERC Discovery grants}
\address{Department of Mathematical and Statistical Sciences, University of Alberta, 632 Central Academic, Edmonton , AB T6G 2R3, Canada}
\email{michaelk@ualberta.ca}
\author[K. L\^e]{Khoa L\^e}
\address{Department of Mathematics, South Kensington Campus, Imperial College London, London, SW7 2AZ, United Kingdom}
\email{n.le@imperial.ac.uk}
\thanks{The second author thanks PIMS for its support through
the Postdoctoral Training Centre in Stochastics during this work}
\author[D. Sezer] {Deniz Sezer}
\address{Department of Mathematics and Statistics, University of Calgary, 2500 University Dr NW, Calgary, AB T2N 1N4, Canada}
\email{adsezer@ucalgary.ca}
\begin{document}
\begin{abstract} A general class of non-Markov, supercritical Gaussian branching particle systems
is introduced and its long-time asymptotics is studied.
Both weak and strong laws of large numbers are developed with
the limit object being characterized in terms of particle motion/mutation.
Long memory processes, like branching fractional Brownian motion and fractional Ornstein-Uhlenbeck
processes with large Hurst parameters, as well as rough processes, like fractional
processes with with smaller Hurst parameter, are included as important examples.
General branching with second moments is allowed and moment measure techniques are utilized.
\end{abstract}
\maketitle
\tableofcontents

\section{Introduction} 
\label{sec:introduction}

Let $\xi=\{\xi_t,t\ge0\}$ be a Gaussian process in $\mathbb R^d$ which starts at $ x$ and has the Volterra representation
\begin{equation}\label{rep:volterra}
	\xi_t=\xi_t(x)=\int_ 0^t K(t,s)dW_s + U_{t} x \,,\ t\ge0,
\end{equation}
where  $W=\{W_s,s\ge0\}$ is a standard Brownian motion in $\RR^d$.
Here, $ x$ is a fixed point in $\RR^d$,  
$t\mapsto U_t$ is a continuous matrix-valued function satisfying
the semigroup property
\(
	U_0y=y\ \text{and } U_tU_sy=U_{t+s}y
\)
for all $s,t\ge 0$ and $y\in\RR^d$, and
$K:\{(t,s)\in(0,\infty)^2:t>s\}\to\RR$ is a measurable kernel such that
\begin{equation*}
	\sigma(t):=\lt(\int_0^t|K(t,s)|^2ds\rt)^{\frac12}
\end{equation*}
is finite for all $t>0$.
Such processes can arise from solving the following stochastic differential equation driven by a fractional Brownian motion $\{B^H_t,t\ge0\}$
\begin{equation}
	d \xi_t=A\xi_tdt+dB^H_t\,,\quad \xi_0= x\,,
\end{equation}
where $H\in(0,1)$ is the Hurst parameter and $A$ is a deterministic matrix.
Throughout, we assume that the representation \eqref{rep:volterra} is canonical, meaning
$\xi$ and $W$ generate the same filtration $\cgg=\{\cgg_t,t\ge0\}$, and refer
to $\xi$ as a Volterra-Gaussian process. 
In particular, fractional Brownian motion is known to be canonical and, therefore, can be considered the proto-typical Volterra-Gaussian process.
Additionally, other common examples include Ornstein-Uhlenbeck and fractional 
Ornstein-Uhlenbeck processes.
Such Volterra-Gaussian processes need not be Markov but rather can have memory.
Still, representation \eqref{rep:volterra} can be used to construct branching processes 
in such a manner that each combined path through all ancestors is a Volterra-Gaussian process
and, given the memory up to a branch time, the motion of the children produced at
that branch time move independently afterwards.
To the authors knowledge, nobody has considered such branching Volterra processes (BVP)
previously so their asymptotic behavior is unknown.
Herein, we study weak and strong laws of large numbers for BVP as time increases to infinity.
The non-Markov memory issue raises questions about: What the appropriate
scaling should be for the anticipation of a non-trivial limit, what the actual limit
should be when the particles can be constrained by long-range dependence
and what methods can be used or adapted in this non-Markov setting.

Interest in spatial branching process asymptotics dates back at least to Watanabe's (1968) study
of branching Markov processes (BMP) and, in particular, branching Brownian motion (BBM)
on $\mathbb R^d$ (\cite{MR0237008}).
He considered the particles alive at a time as a purely-atomic measure, used Fourier techniques
and determined parameters as well as scalings for BBM to avoid local extinction and
have interesting limit behavior.
In particular, he established almost-sure vague convergence in the transient motion case 
to randomly-scaled Lebesgue measure, meaning the mass redistributed itself uniformly
as time goes to infinity regardless how it started but the ``amount" of mass (after scaling) retained
randomness. 
The recurrent motion case has a Gaussian (not Lebesgue) limit and
was treated by Asmussen and Hering \cite{MR0420889} using a different method. 
One of the differences in the two cases is that the rate of convergence is exactly exponential in the later while not so in the former case.
This kind of result has become known as a \emph{strong law of large numbers} and has been extended in various directions by several authors. To name a few, Chen and Shiozawa \cite{MR2352485}, Engl\"ander \cite{MR2500226}, Engl\"ander et al. \cite{MR2641779} have obtained strong laws of large numbers for more general classes of branching Markov processes; Chen et al. \cite{MR2397881}, Engl\"ander \cite{MR2641779},  Wang \cite{MR2644866}, Kouritzin and Ren \cite{MR3131303}, have extended these results
to superprocess limits of branching Markov processes.
In particular, Kouritzin and Ren's result does not require the superprocess to
satisfy the compact support property and strengthens the vague convergence to
so-called shallow convergence, which lies between vague and weak convergence
yet still allows non-finite limits like Lebesgue measure.
Moreover, \cites{MR3395469,MR3010225} have established a key link between strong laws of large numbers
for branching particle systems and those for superprocess limits, using the
idea of skeletons.
It would be interesting to know if their transfer results continue to hold in
our setting where there is memory.

All of the results mentioned above have been for Markov processes.
We study the case where the motion model can have memory, in the Gaussian case.
In particular, we characterize the limit in terms of (the kernel $K$ or) the covariance 
of the Gaussian process and exhibit fractional Brownian motions and
fractional Ornstein-Uhlenbeck processes as our motivating examples.

Our model is motivated in part by Adler and Samorodnitsky's 1995 construction of 
super fractional Brownian motion in \cite{AdlerSamorodnitsky95}, which also supports the 
relevance of our techniques to superprocesses.
To accommodate a variety of branching processes and to ease the transition to superprocesses,
we consider general branching and eschew Fourier techniques in favor of popular superprocess moment-measure techniques. 
We also benefited greatly from the technical branching process
bounding techniques introduced in Asmussen and Hering \cite{MR0420889} and used in Chen and Shiozawa \cite{MR2352485}.
The techniques in this article have been applied to obtain 
asymptotic expansion for supercritical Dawson-Watanabe processes and a certain class Fleming-Viot processes with $\alpha$-stable spatial motions in \cite{Kle1} and \cite{LK} respectively.

We introduce our branching Volterra process model and its atomic measure
notation in the next section.
Section \ref{sec:moment_formulas} of this note contains our basic moment formulae from which
our bounds and asymptotic behavior will later be derived.
Our weak law of large numbers are then established in Section \ref{sec:WLLN}. The proof of our strong law of large number is presented in Section \ref{sec:SLLN}. In Section \ref{sec:examples}, we discuss two key examples of branching particle systems where the particle motions are fractional Brownian motions or fractional Ornstein-Uhlenbeck processes respectively.

\section{Model and main result} 
\label{sec:preliminaries}
We first describe a branching system starting with a single individual particle with \textit{memory}. Let $r$ be a non-negative number.
An individual starting at $\xi_{r}$ will continue the spatial motion from $\xi_0$ to $\xi_{r}$ 
so that the combined trajectory follows the law of the process $\xi$ in \eqref{rep:volterra}. 
We say that this individual has birth memory $\xi[0,r]:= \{\xi_s,0\le s\le r\}$. Her lifetime, $L$, is exponentially distributed with parameter $V>0$ so, 
given that she is alive at time $t$, her probability of dying in the time interval $[t,t+\delta)$ is $V \delta +o(\delta)$. 
When a particle dies at time $d=r+L$, she leaves behind a random number of offspring 
with probability generating function $\psi(s)=\sum_{k=0}^\infty p_k s^k$ and then she is put to the cemetery state $\Lambda$. 
Each of these offspring has birth memory $\xi[0,d]$,   
i.e. the memory starting from $0$ to the time they are born. 
In addition, given the memory of their ancestors, the offspring are independent from each other. 
At each time $t$, the trajectory of each living particle from $0$ to $t$ follows the law of the Volterra-Gaussian process defined in \eqref{rep:volterra}, the locations of individuals who are alive at time $t$ correspond to a measure on $\RR^d$, which is the object of interest. Let us describe it in more detail.

Let $x$ be in $\RR^d$ and $\xi=\xi(x)$ be as defined in \eqref{rep:volterra}. 
We append a cemetery state $\Lambda$ to the state space $\RR^d$ and adopt the convention that $f(\Lambda)=0$ for all functions $f:\RR^d\to\RR$. 
Let $\NN$ ($\NN_0$) be the natural (whole) numbers.
Our branching particle system starts from a single particle with initial memory $\xi[0,r]$, where $r$ is a non-negative number and $\xi_0=x$.
We use multi-indices
\begin{equation*}
\cii:=\{\alpha=(\alpha_0,\dots,\alpha_N): N\in\NN_0, \alpha_j\in \NN\ \forall j\in\NN_0\}\,
\end{equation*} 
to label our particles.
For each $\alpha=(\alpha_0,\dots,\alpha_N)\in\cii$, the \emph{generation} of $\alpha$ is $|\alpha|=N$, $\alpha|_i=(\alpha_0,\dots,\alpha_i)$ with $\alpha|_{-1}=\emptyset$.
We consider an ancestral partial order on $\cii$: $\theta\le \alpha$ iff $\theta$ is an ancestor of $\alpha$ i.e.
\begin{equation*}
\theta\le \alpha \mbox{ if and only if } \theta=\alpha|_i \mbox{ for some } i\le |\alpha|\,.
\end{equation*}
In addition, for two indices $\alpha$ and $\theta$, set
\begin{equation*}
 	|\alpha\wedge \theta|=\max\{k:\alpha_i=\theta_i\quad \forall i\le k\}
\end{equation*}
and write $\alpha\wedge \theta=\alpha|_{|\alpha\wedge \theta|}=\theta|_{|\alpha\wedge \theta|}$, which is the ``greatest common ancestor'' of $\alpha$ and $\theta$.

We will construct our branching particle system using a flow $U$ and a kernel $K$, like
those discussed
in the Introduction, as well as the following family of independent random elements
\begin{equation*}
 	\{\widehat W^\alpha, L^\alpha, S^\alpha:\alpha\in\cii\}
\end{equation*}
on some probability space $(\Omega,\mathcal F,\PP)$.
Here, each $\{\widehat W^\alpha_t,\,t\ge0\}$ is a standard $\RR^d$-valued Brownian motion starting at the origin,
$L^\alpha$ is an exponential random variable with parameter $V$ and 
$S^\alpha$ is an $\NN_0$-valued random variable with probability generating function $\psi$. 
Also, let $S^\emptyset=1$ and
$\cii^l=\{\alpha\in\cii:\alpha_i\le S^{\alpha|_{i-1}}\ \forall i\in\{0,1,...,|\alpha|\}\}$,
the particles that will be born.
The birth time $b^ \alpha$ and death time $d_ \alpha$ of particle $\alpha\in\cii$ are related
by the lifetimes $L^\alpha$ by $d_ \alpha=b^ \alpha+L^\alpha$.
The birth times are defined inductively: $b^{(1)}=r$ and for $\alpha$ such that $|\alpha|\ge1$, $b^ \alpha=d_{\alpha|_{|\alpha|-1}}$.
The driving Brownian motion for particle $\alpha$ is now defined inductively as:
\begin{equation*}
W^{\alpha}_t=\left\{
\begin{array}{ll}
\widehat W^{\alpha}_t & \text{if } |\alpha|=0\\
W^{\alpha|_{|\alpha|-1}}_t & \text{if } t<d_{\alpha|_{|\alpha|-1}},|\alpha|>0\\
W^{\alpha|_{|\alpha|-1}}_{d_{\alpha|_{|\alpha|-1}}}+\widehat W^{\alpha}_{t-d_{\alpha|_{|\alpha|-1}}} &\text{if } t\ge d_{\alpha|_{|\alpha|-1}},|\alpha|>0
\end{array}
\right.
\end{equation*}
and the particle location at time $t$ as:
\begin{equation}\label{partloc}
	\xi_t^\alpha
	=\left\{
	\begin{array}{ll}
\widehat \xi^{\alpha}_{t} & \text{if } b^\alpha \le t< d_\alpha \text{ and }  \alpha\in\cii^l\\
\Lambda &\text{if }  \alpha\notin\cii^l\\
\Lambda &\text{if } t<b^\alpha\text{ or }t\ge d_\alpha
\end{array}
\right.,
\end{equation}
where for every $t\ge0$ and $\alpha$ such that $\alpha_0=1$,
\begin{equation*}
\widehat \xi_t^\alpha=\widehat \xi_t^\alpha(x)=\int_{0}^t K(t,s)dW^\alpha_s+U_{t} x \ \forall t\ge0\,.
\end{equation*}
It is easy to see that each $W^\alpha$ is a standard Brownian motion starting at the origin. Hence, $\widehat \xi^\alpha(x)$ has the same distribution
as $\xi(x)$ defined in (\ref{rep:volterra}) and $\{\xi^\alpha\}_{\alpha\in\cii} $ has all the properties described in the first paragraph in this section.
The first case in (\ref{partloc}) where $\xi_t^\alpha=\Lambda$ corresponds to the situation that the
particle was never alive while the second is that it is not alive at time $t$.
For each $t>0$, the collection of all particles alive at time $t$ is $\cii_t=\{\alpha\in\cii:\xi^\alpha_t\ne \Lambda\}$	and the number of particles alive is $|\cii_t|$, the cardinality of $\cii_t$. 
	The locations of particles alive at time $t$ are stored in an atomic measure
	\begin{equation}\label{def:Xt}
	X_t:= \sum_{\alpha\in\cii_t}\delta_{\xi^\alpha_t}= \sum_{\alpha\in\cii^l}1_{[b^\alpha,d_ \alpha)}(t)\delta_{\xi^\alpha_t},\quad t\ge r \,.
	\end{equation}
The process $X=\{X_t,t\ge r\}$ is called a branching particle system starting from $\xi[0,r]$. Its law is denoted by $\PP_{\xi[0,r]}$. The corresponding expectation is denoted by $\EE_{\xi[0,r]}$.
Finally, if $r=0$ and $\xi_0=x$, we simply write $\PP_x$ and $\EE_x$ respectively.

Let us now describe the particle and branching system information. Let $\cgg^\alpha=\{\cgg^\alpha_t\}_{t\ge 0}$ be the (raw) filtration generated by $\widehat\xi^\alpha$ for each $\alpha\in\cii^l$. 
$\{\cff_t\}_{t\ge 0}$ is the filtration generated by the following random variables:
	\begin{itemize}
	\item $1_{[b^\alpha,d_ \alpha)}(s), \widehat \xi_s^\alpha$ for all $s\in[0, t]$ and $\alpha\in\cii^l$,
		\item $S^\theta$ for all $\theta<\alpha$, where $\alpha\in\cii^l$ such that $1_{[b^\alpha,d_ \alpha)}(s)=1$ for some $s\in[0,t]$,
		\item $b^\theta$ for all $\theta\le\alpha$, where $\alpha\in\cii^l$ such that $1_{[b^\alpha,d_ \alpha)}(s)=1$ for some $s\in[0,t]$.
	\end{itemize}
	We further assume that $\{\cff_t\}_{t\ge 0}$ satisfies the usual condition, i.e. $\{\cff_t\}_{t\ge 0}$ is right-continuous  and $\cff_0$ contains the $\PP$-negligible sets.
	Heuristically, $\cff_t$ includes all information of the branching system $X$ up to time $t$. 
In particular, $X_t$ is $\cff_t$-measurable for all $t\ge r$. 
To ease the notation in Section \ref{sec:WLLN} (to follow), we also take a generic particle $\xi$,
defined in (\ref{rep:volterra}), to be on $(\Omega,\mathcal F,\PP)$ and to generate filtration $\mathcal G$.

We can express $X_t$ (under $\PP_{\xi[0,r]}$) in terms of the particle systems at an earlier time $s$ ($s\ge r$) through the use of memory.
The memory of particle $\alpha$ at time $s$ is given by
\begin{equation*}
\xi^{\alpha}[0,s]=\lt\{\int_{0}^{t}K(t,u)dW^{\alpha}_u+U_{t}x,\ 0\le t\le s \rt\}.
\end{equation*}
(This includes ancestrial memory.)
For each $\alpha\in\cii_s$, let $X^{\alpha,s}=\{X^{\alpha,s}_t,t\ge s\}$ be the branching particle system continuing from a single individual with memory $\xi^\alpha[0,s]$. 
Then, 
\begin{equation}\label{eqn:XtXs}
X_t=\sum_{\alpha\in\cii_s}X^{\alpha,s}_{t}\quad \forall t\ge s\ge r.
\end{equation}
While our motion may be non-Markov, the number of particles counting process, $\{X_t(1),\ t\ge r\}$ under $\PP_{\xi[0,r]}$,	is a continuous time Galton-Watson process starting at time $r$. Given $\xi[0,r]$, its law is independent of the spatial motions $\{\widehat \xi^\alpha\}_{\alpha\in\cii}$. 
We collect some well-known properties of this process in the following remark, denoting by $\beta$  the \textit{branching factor} $\beta=V(\psi'(1^-)-1)$, provided that $\psi'(1^-)$ is finite.

\begin{remark}\label{rem:GW} It follows from classical theory of branching processes (cf. \cite{MR2047480}) that:
\begin{enumerate}[(i)]
\item If $\psi'(1^-)<\infty$, then for all $t\ge r$
\begin{equation*}
	\EE_{\xi[0,r]}X_t(1)=e^{\beta(t-r)}\,.
\end{equation*}
\item If $\psi''(1^-)<\infty$, then for all $t\ge r$
\begin{equation*}
	\EE_{\xi[0,r]}(X_t(1))^2=\lt\{
	\begin{array}{lr}
		e^{\beta(t-r)}+ \psi''(1^-)V\beta^{-1} (e^{2\beta(t-r)}-e^{\beta(t-r)})&\mbox{if }\beta>0\,,
		\\1+\psi''(1^-)V(t-r)&\mbox{if }\beta=0\,.
	\end{array}
	\rt.
\end{equation*}
	In particular, $\sup_{s\in[r,t]}\EE_{\xi[0,r]}(X_s(1))^2<\infty$ for all $t\ge r$.
	\item $\{e^{-\beta (t-r)} X_t(1),t\ge r\}$ is a non-negative $\{\cff_t,t\ge r\}$-martingale with respect to $\PP_{\xi[0,r]}$. 
\end{enumerate}
\end{remark}
Throughout the paper, we assume
\begin{enumerate}[(C0)]
	\item\label{c0} $\psi'(1^-)<\infty$, $\psi''(1^-)<\infty$ and $\beta>0$\,.
\end{enumerate}
Because of the restriction $\beta>0$, our branching system is supercritical. It is known (\cite{MR2047480}) that under \ref{c0}, there exists a non-trivial random variable $F$ such that $\PP_{\xi[0,r]}$-a.s.
\begin{equation}\label{def:F}
	\lim_{t\to\infty}e^{-\beta t}X_{t}(1)=e^{-\beta r} F\,.
\end{equation}
Note that in general, $F$ may depend on $\xi[0,r]$. However, given $\xi[0,r]$, $F$ is independent of the spatial motions $\{\widehat \xi^\alpha\}_{\alpha\in \cii}$.

Next, we are going to describe our main result. To do this, we need to introduce some more notation. We consider a branching system $X$ starting with a single memory $\xi[0,r]$, $r\ge0$.
We remark that the memory $\xi[0,r]$ itself is a trajectory of a stochastic process, which can  behave oddly. For our purpose, the following condition on the memories are enforced throughout.
\begin{description}
	\item[Typical memory] $\xi[0,r]$ is \textit{typical} in the sense that its driving Brownian motion $W$ satisfies
\begin{equation}\label{TypMem}
		\lim_{t\to\infty}\frac{|\int_0^r K(t,u)dW_u| }{\lt(\int_0^t|K(t,u)|^2du\rt)^{1/2}}=0\,.
	\end{equation}
\end{description}
Note that memories of zero length are typical in the above sense.
For each $t\ge s>0$, the random variables $\int_0^s K(t,u)dW_u$ and $\int_s^tK(t,u)dW_u$ are centered normal random variables with standard deviations
\begin{equation}
\sigma_1(t,s)=\lt(\int_0^s|K(t,u)|^2du \rt)^{\frac12}\,\mbox{ and }\,\sigma_2(t,s)=\lt(\int_s^t|K(t,u)|^2du \rt)^{\frac12}
\end{equation}
(so $\sigma^2(t)=\sigma^2_1(t,s)+\sigma^2_2(t,s)$ for all $0\le s\le t$) and densities 
\begin{equation}\label{eqn_pj}
p_j(t,s,x)=(2 \pi \sigma_j^2(t,s))^{-\frac d2}\exp\lt\{-\frac{|x|^2}{2 \sigma^2_j(t,s)}\rt\}\,\text{for }j=1,2.
\end{equation}
To describe our strong law of large numbers,  we consider the following conditions:
\begin{enumerate}[(C\arabic*)]
\item\label{c1} $\displaystyle \lim_{t\to\infty}e^{-\beta t}\sigma^d(t)=0$ and $\displaystyle\lim_{t\to\infty}\sigma(t) =\ell$ for some $\ell\in(0,\infty]$.
	\item \label{con:U} There exists a $d\times d$-matrix $U_\infty$ such that $\displaystyle\lim_{t\to\infty}\frac{U_tx}{\sigma(t)}=U_\infty x$ for every $x\in\RR^d$.
	\item\label{con:s1s} There exists a function $b: (0,\infty)\rightarrow (0,\infty)$ such that $\lim_{t\rightarrow\infty}e^{-\beta  b(t)}\sigma^d(t) =0$, $b(t)<t$ for $t$ sufficiently large, and 
	\[\lim_{t\to\infty}\frac{\sigma_1(t,b(t))\sqrt{\ln t} }{\sigma(t) }=0.\]
	\item \label{con:sl.tnmax} Let $\xi[0,r]$ be our fixed typical memory with $r\ge0$. There exists an increasing sequence $\{t_n\}$ in $(r,\infty)$ which satisfies
	\begin{equation}\label{limtn}
	 	\lim_{n\to\infty}(t_{n+1}-t_n)=0\,,\quad \lim_{n\to\infty}\frac{\sup_{t\in[t_n,t_{n+1}]}\sigma(t)}{\sigma(t_n)}=\lim_{n\to\infty}\frac{\inf_{t\in[t_n,t_{n+1}]}\sigma(t)}{\sigma(t_n)}=1\,,
\end{equation} 
\begin{equation}\label{sumtn}
t_n\ge n^\gamma\ \text{ for some }\gamma>0,\ \ \ \sum_{n=2}^\infty e^{-\beta b(t_{n-1})}\sigma^d(t_n)<\infty\,,
\end{equation}
and
\begin{equation}\label{sumPmax}
	\sum_{n=1}^\infty \sigma^{d}(t_n)\PP_{\xi[0,r]}\lt(\sup_{u,v\in[t_n,t_{n+1}]}\lt|\xi_u- \xi_v \rt|\ge \epsilon \rt)<\infty
\end{equation}
	for every $\epsilon>0$ sufficiently small.
\end{enumerate}
We emphasize that $U_\infty$ is the limit of $U_t$ normalized by $\sigma(t)$ as $t\to\infty$. An example of a Gaussian process satisfying \ref{con:U} is the process
\begin{equation}\label{exp1}
	\xi_t=e^t \xi_0+\int_0^t e^sdW_s \quad\forall t\ge0\,.
\end{equation}
We will see at the end of Section \ref{sec:examples} that  branching systems corresponding to the above process with sufficient large branching factors satisfy the weak law of large numbers.

Our main result is the following theorem. 
\begin{theorem}\label{thm:SLLN}
 	Let $\{X_t\}_{t\ge r}$ be a branching system starting with  typical memory $\xi[0,r]$ of length $r\ge0$. Let $F$ be the random variable in \eqref{def:F}. Assume that conditions \ref{c0}-\ref{con:sl.tnmax} are satisfied. With $\PP_{\xi[0,r]}$-probability one, we have
\begin{equation}\label{eqn:slln}
\lim_{t\to\infty}e^{-\beta t}\sigma^d(t)X_t(g)=e^{-\beta r} F (2 \pi)^{-\frac d2}\int_{\RR^d}\exp\lt\{-\frac12\lt|U_\infty \xi_0-\frac{y}\ell \rt|^2\rt\}g(y)dy
\end{equation}
for every bounded continuous function $g:\RR^d\to\RR$ if $\ell$ is finite, and for every continuous function $g:\RR^d\to\RR$ such that $\sup_{x\in\RR^d}e^{\epsilon|x| }|g(x)|<\infty$ for some $\epsilon>0$ if $\ell$ is infinite.
\end{theorem}
For clarity, in the case $\ell=\infty$, we take $\frac{y}\ell=0$ and (\ref{eqn:slln}) simplifies.

The following result is an immediate consequence.
\begin{corollary}
	Under the hypothesis of Theorem \ref{thm:SLLN}, we assume in addition that $\ell$ is finite. Then with $\PP_{\xi[0,r]}$-probability one, as $t\to\infty$ the measure $e^{-\beta (t-r)}X_t$ converges weakly to the measure $F(2 \pi\ell^2)^{-\frac d2}\exp\{-\frac12|U_\infty \xi_0-\frac{y}\ell|^2 \}dy$, which is a multiple of $F$ and a Gaussian measure. 
\end{corollary}
It is worth noting that the random variable $F$ in \eqref{eqn:slln} is independent of the spatial motions $\{\xi^\alpha\}_{\alpha\in\cii}$. Only two characteristics of the spatial motions are reflected in the limiting object, namely $\ell$ and $U_\infty$. If the spatial motions are either fractional Brownian motions or fractional Ornstein-Uhlenbeck processes, $U_\infty\equiv0$ (see Section \ref{sec:examples} for details).

Let us explain the heuristic reasons behind \eqref{eqn:slln}. Suppose for simplicity that $g$ is a continuous function with compact support. From \eqref{def:Xt}, we can write
\begin{equation*}
	X_t(g)=\sum_{\alpha\in\cii_t}g(\xi_t^\alpha)\,.
\end{equation*}
As $t\to\infty$, $|\cii_t|=X_t(1)$ tends to infinity with rate $e^{\beta t}$. 
Hence, one anticipates a large number effect: 
\begin{equation*}
	\frac1{X_t(1)}\sum_{\alpha\in\cii_t} g(\xi_t^\alpha)\approx\lt[\EE_{\xi[0,r]} g(\xi_t)\rt]\,\text{as }t\to\infty.
\end{equation*}
We observe that $e^{-\beta t}X_t(1)$ converges to $e^{-\beta r}F$ and
\begin{align}
	&\lim_{t\to\infty}\sigma^d(t) \EE_{\xi[0,r]} g(\xi_t)
	\nonumber\\&=\lim_{t\to\infty}(2 \pi)^{-\frac d2}\frac{\sigma^d(t)}{\sigma^d_2(t,r)} \int_{\RR^d}\exp\lt\{-\frac12\lt|\frac{\int_0^rK(t,u)dW_u}{\sigma_2(t,r)} -\frac{U_t \xi_0}{\sigma_2(t,r)}-\frac1{\sigma_2(t,r)}y\rt|^2 \rt\}g(y)dy
	\label{lim:Efxi}
\end{align}
As we will see later in Remark \ref{rem:ss}, due to \ref{con:s1s}, we can replace $\sigma_2(t,r)$ by $\sigma(t)$ without changing the value of the limit. The first term inside the exponential vanishes since $\xi[0,r]$ is a typical memory. The limits of the remaining terms are asserted by the conditions \ref{c1} and \ref{con:U}. Hence, the limit in \eqref{lim:Efxi} is
\begin{equation*}
	(2 \pi)^{-\frac d2}\int_{\RR^d}\exp\lt\{-\frac12\lt|U_\infty \xi_0-\frac1{\ell}y\rt|^2 \rt\}g(y)dy\,.
\end{equation*}
One discovers the identity \eqref{eqn:slln}. An important observation from this argument is that the long term dynamic of $X_t(g)$ is factorized into the long term dynamics of $X_t(1)$ and $\EE_{\xi[0,r]}g(\xi_t)$.

We note that condition \ref{con:sl.tnmax} is not employed in the above heuristic argument. 
It is not clear whether Theorem \ref{thm:SLLN} can be proved without \ref{con:sl.tnmax}. A rigorous proof of Theorem \ref{thm:SLLN} is presented in Section \ref{sec:SLLN}, where we first obtain almost convergence along a suitable sequence $\{t_n\}$, then condition \ref{con:sl.tnmax} is used to transfer the convergence to continuous time.

Throughout the paper the notation $A\lesssim B$ means $A\le CB$ for some constant $C>0$. 

\section{Moment formulae} 
\label{sec:moment_formulas}

In this section, we will develop the moment formulae for our BVPs that will
be the basic tools for our laws of large numbers to follow in later sections. 
In what follows, let us fix $t\ge s\ge r\ge 0$ and consider a branching system $X$ starting with a memory $\xi[0,r]$. We denote $\EE_{\xi[0,r],s}=\EE_{\xi[0,s]}(\,\cdot\,|\cff_s)$.
We are interested in explicit formulae for the following quantities
\begin{equation*}
	m_ \phi(t,\xi[0,r]):= \EE_{\xi[0,r]}\sum_{\alpha\in\cii_t} \phi(\xi^\alpha)\,,\quad\mbox{and}\quad
	 m_{f_1,f_2}(t,\xi[0,r])=\EE_{\xi[0,r]}[X_t(f_1)X_t(f_2)]\,,
\end{equation*}
where $\phi:(\RR^d)^{[0,\infty)}\to\RR$ is a measurable functional and $f_1,f_2:\RR^d\to\RR$ are measurable functions. 
In the special case when $\phi(\xi)=f(\xi_t)$, we have $m_ \phi(t,\xi[0,r])= m_f(t,\xi[0,r])=\EE_{\xi[0,r]}X_t(f)$.
\begin{lemma}\label{lem:firstmom}
	Assume that $\psi'(1^-)<\infty$. For every measurable functional $\phi:(\RR^d)^{[0,\infty)}\to\RR$ which is either non-negative or satisfies
	\begin{equation*}
		\EE_{\xi[0,r]}|\phi(\xi)|<\infty\,,
	\end{equation*}
	we have
	\begin{equation}\label{eqn:m1}
		m_ \phi(t,\xi[0,r])=e^{\beta(t-r)}\EE_{\xi[0,r]}\phi(\xi)\,.
	\end{equation}
\end{lemma}
\begin{proof}
Similar to \eqref{def:Xt}, we can write
\begin{equation*}
	m_ \phi(t,\xi[0,r])=\EE_{\xi[0,r]}\sum_{\alpha\in\cii^l}1_{[b^\alpha,d_ \alpha)}(t)\phi(\xi^\alpha)\,.
\end{equation*}
Since $\cii^l$ depends only on $\mathcal{S} :=\{S^\alpha:\alpha\in\cii\}$, conditioning on $\mathcal S$ and using independency yield
\begin{align*}
	m_ \phi(t,\xi[0,r])&=\EE_{\xi[0,r]}\sum_{\alpha\in\cii^l}\EE_{\xi[0,r]}\lt(1_{[b^\alpha,d_ \alpha)}(t)\phi(\xi^\alpha)\Big| \mathcal S\rt)
	\\&= \EE_{\xi[0,r]}\sum_{\alpha\in\cii^l}1_{[b^\alpha,d_ \alpha)}(t)\EE_{\xi[0,r]}\phi(\xi^\alpha)\,.
\end{align*}
Thanks to Fubini and Tonelli's theorems, we have freely interchanged the order of expectation and summation.
For each $\alpha$ which is alive at time $t$, from \eqref{partloc}, we see that $\xi^\alpha$ follows the law of $\xi$ defined in \eqref{rep:volterra}, thus
\begin{equation*}
	\EE_{\xi[0,r]}\phi(\xi^\alpha)=\EE_{\xi[0,r]}\phi(\xi)\,.
\end{equation*}
Hence, in conjunction with the previous identity of $m_ \phi$, we obtain
\begin{equation}
	m_ \phi(t,\xi[0,r])=\EE_{\xi[0,r]}\phi(\xi)\EE_{\xi[0,r]}\sum_{\alpha\in\cii^l}1_{[b^\alpha,d_ \alpha)}(t)\,.
\end{equation}
The last expectation on the right-hand side above is $\EE_{\xi[0,r]}X_t(1)=e^{\beta(t-r)}$ by point (i) of Remark \ref{rem:GW}.
\end{proof}
We will also utilize an explicit formula for the second moments.
\begin{lemma}
Suppose that $\psi''(1^-)<\infty$. 
Then, for all functions $f_1,f_2$ in $L^2(\RR^d)$,
\begin{multline}\label{eqn:m2}
\EE_{\xi[0,r]}(X_t(f_1)X_t(f_2))
=\EE_{\xi[0,r]}\lt[f_1(\xi_{t})f_2(\xi_{t})\rt]e^{\beta (t-r)}
\\+V\psi''(1^-)e^{2 \beta (t-r)} \int_r^t\EE_{\xi[0,r]}\lt[\EE_{\xi[0,r],u}(f_1(\xi_{t}))\EE_{\xi[0,r],u}(f_2(\xi_{t})) \rt]e^{- \beta (u-r)}du\,.
\end{multline}
\end{lemma}
\begin{remark}
(\ref{eqn:m2}) can be given a genealogical interpretation.
Writing
\begin{eqnarray*}
X_t(f_1)X_t(f_2)&=&\sum_{\alpha,\alpha'\in \mathcal I_t}f_1(\xi^\alpha_t)
f_2(\xi^{\alpha'}_t)
\\
&=&\sum_{\alpha\in \mathcal I_t}f_1(\xi^\alpha_t)
f_2(\xi^{\alpha}_t)+\sum_{\alpha\ne\alpha'\in \mathcal I_t}f_1(\xi^\alpha_t)
f_2(\xi^{\alpha'}_t)
\end{eqnarray*}
and then taking expectations,
we can think of the second term in the last expression of (\ref{eqn:m2})
as decomposing $\EE_{\xi[0,r]}\lt[\sum_{\alpha\ne\alpha'\in \mathcal I_t}f_1(\xi^\alpha_t)
f_2(\xi^{\alpha'}_t) \rt]$
over the time $u$ when the last common ancestor of $\alpha,\alpha'$ died.
This will become clearer in the following proof.
\end{remark}
\begin{proof}
We assume first that $f_1,f_2$ are bounded. In this case, it follows by Remark \ref{rem:GW} (ii) that
\begin{equation}\label{tmp:supm2}
	\sup_{v\in[r,t]}|m_{f_1,f_2}(v,\xi[0,r])|\lesssim\sup_{v\in[r,t]} \EE_{\xi[0,r]}[(X_v(1))^2] <\infty\quad\mbox{for all}\quad t\ge r\,.
\end{equation}
Let $\zeta$ and $\xi$ be the first branching time (after $r$) and the trajectory of the first individual. Using the independence between $\zeta$ and the spatial motions, we have
	\begin{equation*}
	\EE_{\xi[0,r]}[X_t(f_1)X_t(f_2)1_{t<\zeta}]=\EE_{\xi[0,r]}[f_1(\xi_t)f_2(\xi_t)]e^{-V(t-r)}\,.
	\end{equation*}
	Conditioning on $\zeta$ and $\xi$, we see that
	\begin{equation*}
	\EE_{\xi[0,r]}[X_t(f_1)X_t(f_2)1_{t>\zeta}]
	=\int_r^t\EE_{\xi[0,r]}\left[\sum_{k=0}^\infty p_k\sum_{i,j=1}^k \EE_{\xi[0,r],u}\left( X_{t}^{i,u}(f_1)X_{t}^{j,u}(f_2) \right) \right]Ve^{-V(u-r)}du,
	\end{equation*}
	where $k$ represents the number of offspring produced at time $\zeta$, and 
	for each $i=1,\dots, k$, $X^{i,u}=\{X^{i,u}_t,t\ge u\}$ is the branching system starting from the 
	$i$-th offspring. Conditional on $\cff_u$, each $X^{i,u}$ has memory $\xi[0,u]$ and is independent from each other. 
	By considering two cases when $i=j$ and $i\neq j$, the right hand side above is the same as
	\begin{multline*}
	\int_r^t \EE_{\xi[0,r]} \left[\psi'(1^-)m_{f_1,f_2}(t,\xi[0,u]) \right]Ve^{-V(u-r)}du
	\\+\int_r^t \EE_{\xi[0,r]} \left[\psi''(1^-)m_{f_1}(t,\xi[0,u])m_{f_2}(t,\xi[0,u])\right]Ve^{-V(u-r)} du\,.
	\end{multline*}
	In addition, from \eqref{eqn:m1}
	\begin{equation*}
		\EE_{\xi[0,r]} \left[\psi''(1^-)m_{f_1}(t,\xi[0,u])m_{f_2}(t,\xi[0,u])\right]=\psi''(1^-)e^{2 \beta(t-u)}\EE_{\xi[0,r]}g(t,u)
	\end{equation*}
	where $g(t,u)=\EE_{\xi[0,r],u}(f_1(\xi_t))\EE_{\xi[0,r],u}(f_2(\xi_t))$.
	Altogether, we see that $m_{f_1,f_2}$ satisfies the equation
	\begin{multline}\label{eqn:m2eqnsimp}
	m_{f_1,f_2}(t,\xi[0,r])-\psi'(1^-)V\int_r^t \EE_{\xi[0,r]} \left[m_{f_1,f_2}(t,\xi[0,v]) \right]e^{-V(v-r)}dv
	\\=\EE_{\xi[0,r]}[f_1(\xi_t)f_2(\xi_t)]e^{-V(t-r)}
	+\psi''(1^-)V\int_r^t \EE_{\xi[0,r]}[g(t,u)]e^{2 \beta(t-u)} e^{-V(u-r)}du\,.
	\end{multline}
	We now check that the right-hand side \eqref{eqn:m2} is a solution to \eqref{eqn:m2eqnsimp}. Indeed, under \eqref{eqn:m2}, the left-hand side of \eqref{eqn:m2eqnsimp} becomes
	\begin{multline*}
	 	\EE_{\xi[0,r]}[f_1(\xi_t)f_2(\xi_t)]e^{\beta(t-r)}+\psi''(1^-)Ve^{2\beta(t-r)} \int_r^t\EE_{\xi[0,r]}[g(t,u)]e^{-\beta (u-r)}du
	 	\\ -\psi'(1^-)V\int_r^t \EE_{\xi[0,r]}[f_1(\xi_t)f_2(\xi_t)]e^{\beta(t-v)} e^{-V(v-r)}dv
	 	\\- \psi'(1^-)\psi''(1^-)V^2\int_r^t\int_v^t\EE_{\xi[0,r]}[g(t,u)]e^{2 \beta(t-v)}e^{-V(v-r)}e^{-\beta(u-v)}dudv\,,
	 \end{multline*} 
	which coincides with the right-hand side of \eqref{eqn:m2eqnsimp} after some integral computations.
	On the other hand, equation \eqref{eqn:m2eqnsimp} has at most one solution satisfying \eqref{tmp:supm2}.
	In fact, for every fixed $t\ge0$, the difference $\Delta_r$ of two solutions of (\ref{eqn:m2eqnsimp}) would have to satisfy 
\[
|\Delta_r|\le \psi'(1^-)V\int_r^t\EE_{\xi[0,r]}\left(|\Delta_v| \right)e^{-V(v-r)}dv\ \forall r\le t.
\]
Iterating this inequality yields
\begin{equation*}
	|\Delta_r|\le \lt(\sup_{v\in [r,t]}\EE_{\xi[0,r]}(|\Delta_v|)\rt) \frac{(\psi'(1^-)V(t-r))^n}{n!}
\end{equation*}
for every $r\le t$ and $n\in\NN$. Sending $n$ to infinity implies $\Delta_r=0$ for all $r\le t$. 
Hence, equation \eqref{eqn:m2eqnsimp} has at most one solution, which is given by the right-hand side of \eqref{eqn:m2}.

To take off the boundedness restriction, we note by Jensen's inequality, for every $t\ge r$ 
\begin{equation*}
 	\EE_{\xi[0,r]}\lt[\EE_{\xi[0,r],u}|f_1(\xi_t)|\EE_{\xi[0,r],u}|f_2(\xi_t)|\rt]\le\EE_{\xi[0,r]}(|f_1|\vee |f_2|)^2(\xi_t)\lesssim  \||f_1|\vee |f_2|\|^2_{L^2(\RR^d)}\,.
\end{equation*} 
Hence, for general functions $f_1,f_2$ in $L^2(\RR^d)$, we can extend identity \eqref{eqn:m2} for bounded functions to $L^2(\RR^d)$-functions using truncation and Lebesgue's dominated convergence theorem.
\end{proof}

As applications, we derive the following two propositions which are essential in our approach. 
\begin{proposition}\label{prop:XFs}
	Suppose that $\psi'(1^-)<\infty$ and $r\le s\le t$. For every measurable function $f:\RR^d\to\RR$ which is either non-negative or satisfies $\EE_{\xi[0,r]}|f(\xi_t)|<\infty$, we have
	\begin{equation}
		\EE_{\xi[0,r]}\lt(X_t(f)\big|\cff_s \rt)=e^{\beta(t-s) }\sum_{\alpha\in\cii_s} \EE_{\xi[0,r]}\lt(f(\xi_t^\alpha)\big|\cgg_s^\alpha \rt)
	\end{equation}
\end{proposition}
\begin{proof}
	From \eqref{eqn:XtXs} we have
	\begin{equation*}
		\EE_{\xi[0,r]}(X_t(f)|\cff_s)=\sum_{\alpha\in\cii_s}\EE_{\xi[0,r]}(X_t^{\alpha,s}(f)|\cff_s)\quad \PP_{\xi[0,r]}-\textrm{a.s.}
	\end{equation*}
	Conditional on $\cff_s$, for each $\alpha\in\cii_s$, $X^{\alpha,s}$ is a branching system starting from the memory $\xi^\alpha[0,s]$. Thus, Lemma \ref{lem:firstmom} is applied to get
	\begin{equation*}
		\EE_{\xi[0,r]}(X_t^{\alpha,s}(f)|\cff_s)=m_f(t,\xi^\alpha[0,s]) =e^{\beta(t-s)}\EE_{\xi[0,r]}(f(\xi^\alpha_t)|\cgg^\alpha_s)\,.
	\end{equation*}
	The result follows upon combining the previous identities.
\end{proof}
\begin{proposition}\label{prop:PXtFt0}
For every $t\ge s\ge r\ge 0$ and $f\in L^2(\RR^d)$, we have
\begin{align}\label{eqn:Ptt02}
	\EE_{\xi[0,r]}&\lt(\lt(X_t(f)-\EE_{\xi[0,r]}(X_t(f)|\cff_{s})\rt)^2 \rt)
	\nonumber\\&=e^{\beta (t-r)}\EE_{\xi[0,r]}\lt(f^2(\xi_{t})\rt)-e^{2 \beta t- \beta s- \beta r}\EE_{\xi[0,r]} \lt(\EE_{\xi[0,r],s}f(\xi_t) \rt)^2
	\\&\quad+e^{2 \beta t- \beta r}\int_s^{t}V \psi''(1^-) \EE_{\xi[0,r]}\lt(\EE_{\xi[0,r],u}f(\xi_t)\rt)^2e^{-\beta u}du.\nonumber
\end{align}
In particular, there exists a constant $C=C(V,\psi''(1^-), \beta)$ such that
\begin{equation}\label{est:X|Fs}
	\EE_{\xi[0,r]}\lt(\lt(X_t(f)-\EE_{\xi[0,r]}(X_t(f)|\cff_{s})\rt)^2 \rt)\le C  e^{2 \beta t- \beta s- \beta r}\sigma_2^{-d}(t,r)\|f\|_{L^2(\RR^d)}^2\,.
\end{equation}
\end{proposition}
\begin{proof}
From \eqref{eqn:XtXs} we have 
\begin{equation*}
	X_t(f)-\EE_{\xi[0,r]}(X_t(f)|\cff_{s})=\sum_{\alpha\in \cii_{s} }\lt[X_t^{\alpha,s}(f)-\EE_{\xi[0,r]}(X_t^{\alpha,s}(f)|\cff_{s})\rt]\,.
\end{equation*}
Hence, 
\begin{align*}
	&\EE_{\xi[0,r]}\lt(\lt(X_t(f)-\EE_{\xi[0,r]}(X_t(f)|\cff_{s})\rt)^2 \big|\cff_{s}\rt)
	\\&=\sum_{\alpha\in \cii_{s} }\EE_{\xi[0,r]}\lt(\lt[X_t^{\alpha,s}(f)-\EE_{\xi[0,r]}(X_t^{\alpha,s}(f)|\cff_{s})\rt]^2\big|\cff_{s}\rt)
	\\&+\sum_{\alpha,\alpha'\in\cii_{s};\alpha\neq {\alpha'}}\EE_{\xi[0,r]}\lt([(X_t^{\alpha,s}(f)-\EE_{\xi[0,r]}(X_t^{\alpha,s}(f)|\cff_{s})][(X_t^{\alpha',s}(f)-\EE_{\xi[0,r]}(X_t^{\alpha',s}(f)|\cff_{s})]\big|\cff_{s}\rt).
\end{align*}
We note that under $\PP_{\xi[0,r]}$, conditional on $\cff_{s}$, the branching systems $X^{\alpha,s}$, $\alpha\in\cii_{s}$ are independent from each other. Thus the second sum above vanishes. We arrive at
\begin{multline}\label{tmp:e1}
	\EE_{\xi[0,r]}\lt(\lt(X_t(f)-\EE_{\xi[0,r]}(X_t(f)|\cff_{s})\rt)^2 \rt)
	\\=\EE_{\xi[0,r]}\lt(\sum_{\alpha\in \cii_{s} }\EE_{\xi[0,r]}\lt(\lt[X_t^{\alpha,s}(f)-\EE_{\xi[0,r]}(X_t^{\alpha,s}(f)|\cff_{s})\rt]^2\Big|\cff_{s}\rt)\rt)\,.
\end{multline}
Now, for each $\alpha\in\cii_s$, 
\begin{align*}
	\EE_{\xi[0,r]}&\lt(\lt[X_{t}^{\alpha,s}(f)-\EE_{\xi[0,r]}(X_{t}^{\alpha,s}(f)|\cff_{s})\rt]^2\Big|\cff_{s}\rt)
	\\&=\EE_{\xi[0,r]}\lt[(X_t^{\alpha,s}(f))^2|\cff_s\rt]-\lt[\EE_{\xi[0,r]}(X_t^{\alpha,s}(f)|\cff_s)\rt]^2
	\\&=m_{f,f}(t,\xi^\alpha[0,s]) -\lt[m_f(t,\xi^\alpha[0,s]) \rt]^2=: \phi(\xi^\alpha[0,s])\,.
\end{align*}
Hence, from \eqref{tmp:e1}, applying \eqref{eqn:m1}, we have
\begin{equation*}
	\EE_{\xi[0,r]}\lt(\lt(X_t(f)-\EE_{\xi[0,r]}(X_t(f)|\cff_{s})\rt)^2 \rt)
	=\EE_{\xi[0,r]}\sum_{\alpha\in\cii_s}\phi(\xi^\alpha[0,s])
	=e^{\beta(s-r)}\EE_{\xi[0,r]}\phi(\xi[0,s])\,.
\end{equation*}
Now \eqref{eqn:m1} and \eqref{eqn:m2} give
\begin{align*}
	\EE_{\xi[0,r]} \phi(\xi[0,s])
	&=e^{\beta(t-s)}\EE_{\xi[0,r],s}f^2(\xi_{t})-e^{2 \beta(t-s)}\EE_{\xi[0,r]}\lt(\EE_{\xi[0,r],s}f(\xi_t) \rt)^2
	\\&\quad+e^{2 \beta(t-s)}\int_s^{t}V \psi''(1^-) \EE_{\xi[0,r]}\lt(\EE_{\xi[0,r],u}f(\xi_t)\rt)^2e^{-\beta(u-s)}du\,.
\end{align*}
Upon combining  these equalities together, we arrive at \eqref{eqn:Ptt02}. 

Applying Jensen's inequality and recalling the definition of $p_2$ from (\ref{eqn_pj}), we see that for every $t\ge u\ge r$
\begin{align*}
	\EE_{\xi[0,r]}\lt(\EE_{\xi[0,r],u}f(\xi_t) \rt)^2&\le \EE_{\xi[0,r]}f^2(\xi_t)
	\\&=\int_{\RR^d}f^2\lt(U_t x+\int_0^rK(t,u)dW_u+y\rt)p_2(t,r,y)dy
	\\&\le (2 \pi \sigma^2_2(t,r))^{-\frac d2}\int_{\RR^d}f^2(y)dy\,.
\end{align*}
From here, \eqref{est:X|Fs} follows easily.
\end{proof}

\section{The weak law of large numbers} 
\label{sec:WLLN}

We study convergence in probability of \eqref{eqn:slln} for a fixed test function $g$. 
This is weaker than the almost sure convergence asserted in Theorem \ref{thm:SLLN}. We include this result here because  the assumptions of Theorem \ref{thm:SLLN} are relaxed and the proof of convergence in probability is considerably simpler. In particular, condition \ref{con:sl.tnmax} is not needed and  condition \ref{con:s1s} can be replaced by a milder condition.
	\begin{enumerate}[(C3')]
	 	\item\label{con:ws1s} There exists a function $b: (0,\infty)\rightarrow (0,\infty)$ such that $\lim_{t\rightarrow\infty}e^{-\beta b(t)}\sigma^d(t) =0$, $b(t)<t$ for $t$ sufficiently large, and 
		\[\lim_{t\to\infty}\frac{\sigma_1(t,b(t))}{\sigma(t) }=0.\]
	\end{enumerate}
	\begin{remark}\label{rem:ss} (i) Condition \ref{con:s1s} implies \ref{con:ws1s}. 

	\noindent (ii) Condition \ref{con:ws1s} implies
	\begin{equation}
		\lim_{t\to\infty}\frac{\sigma_2(t,b(t))}{\sigma(t)}=1
	\end{equation}
	and
	\begin{equation}
		\lim_{t\to\infty}\frac{\sigma_1(t,r)}{\sigma(t)}=0
	\end{equation}
	for every $r>0$. These are evident since $\sigma_1^2(t,b(t))+\sigma_2^2(t,b(t))=\sigma^2(t)$ and $\sigma_1(t,r)\le \sigma_1(t,b(t))$ for all $t$ sufficiently large.
	\end{remark}
	The following theorem is the main result of the current section.
\begin{theorem}\label{thm:wlln}
Let $X$ be a branching system starting from a typical memory $\xi[0,r]$ with $r\ge0$. We assume that conditions \ref{c0}-\ref{con:U} and \ref{con:ws1s} are satisfied.
Then for every measurable function $f$ in $L^1(\RR^d)\cap L^2(\RR^d)$, 
\begin{equation}\label{eqn:4.3}
	\lim_{t\to\infty}e^{-\beta t}\sigma^d(t)X_t(f)
	=e^{-\beta r} (2 \pi)^{-\frac d2}F \int_{\RR^d}\exp\lt\{-\frac12\lt|U_\infty \xi_0-\frac{y}\ell \rt|^2\rt\}f(y)dy
\end{equation}
in $\PP_{\xi[0,r]}$-probability, where $F$ is the random variable defined in \eqref{def:F}.
	\end{theorem} 
	The proof of this theorem is postponed to the end of the current section. Since our limit result is not at the level of measure-valued process, it is an abuse of terminology to call the above theorem weak law of large numbers.
	For convenience, for each function $f$, let $\mathcal T f$ be the function defined by
\begin{equation}\label{def:Tf}
\mathcal Tf(z)=(2 \pi)^{-\frac d2}\int_{\RR^d}\exp\lt\{-\frac12\lt|z-\frac{y}\ell \rt|^2\rt\}f(y)dy\,.
\end{equation}
The proof of Theorem \ref{thm:wlln} undergoes two main steps of showing
\begin{equation}\label{lim:step1}
	\lim_{t\to\infty}e^{-\beta t}\sigma^d(t)X_t(f)=\lim_{t\to\infty}e^{-\beta t}\sigma^d(t)\EE_{\xi[0,r]}( X_t(f)|\cff_{b(t)})
\end{equation}
and
\begin{equation}\label{lim:step2}
		\lim_{t\to\infty}e^{-\beta t}\sigma^d(t)\EE_{\xi[0,r]}( X_t(f)|\cff_{b(t)})=\lim_{t\to\infty}e^{-\beta b(t)}X_{b(t)}(1) \mathcal Tf(U_\infty \xi_0)\,,
\end{equation}
where the convergences are in $\PP_{\xi[0,r]}$-probability.
These are accomplished through the following lemmas. The first one is an extension of \eqref{lim:Efxi}.
\begin{lemma}\label{lem:ergodic}  Under the assumptions \ref{c1}, \ref{con:U} and \ref{con:ws1s},
	\begin{equation}\label{eq:cme}
		\lim_{t\to\infty}\EE_{\xi[0,r]}\lt|\sigma^d(t)\EE_{\xi[0,r]}\lt(f(\xi_t)|\cgg_{b(t)}\rt)- \mathcal T f(U_\infty \xi_0) \rt|=0
	\end{equation}
for any $f$ in $L^1(\RR^d)$.
\end{lemma}
\begin{proof}
	Using explicit densities of normal random variables, we have
	\begin{align*}
	&\EE_{\xi[0,r]}\lt|\sigma^d(t)\EE_{\xi[0,r]}\lt(f(\xi_t)|\cgg_{b(t)} \rt)- \mathcal T f(U_\infty \xi_0) \rt|
\\&\lesssim \EE_{\xi[0,r]}\int_{\RR^d}|f(y)|\lt|\frac{\sigma^d(t)}{\sigma_2^d(t,b(t))} e^{-\frac{|y-U_t\xi_0-\int_0^{b(t)}K(t,u)dW_u|^2}{2 \sigma_2^2{(t,b(t))}}} -e^{-\frac12\lt|U_\infty \xi_0-\frac y\ell\rt|^2} \rt|dy
\\&\lesssim\int_{\RR^d}\int_{\RR^d}|f(y)|\lt|\frac{\sigma^d(t)}{\sigma_2^d(t,b(t))} e^{-\frac{|y-U_t\xi_0-\int_0^{r}K(t,u)dW_u-z \int_r^{b(t)}|K(t,u)|^2du |^2}{2 \sigma_2^2{(t,b(t))}}} -e^{-\frac12\lt|U_\infty \xi_0-\frac y\ell\rt|^2} \rt| dy e^{-\frac{|z|^2}2 } dz\,.
		\end{align*}
We note that by \ref{con:ws1s} and Remark \ref{rem:ss},
\begin{equation*}
	\lim_{t\to\infty}\frac{\sigma^d(t)}{\sigma_2^d(t,b(t))}= 1\quad\mbox{and}\quad 
\lim_{t\to\infty}\frac{\int_r^{b(t)}|K(t,u)|^2du}{\sigma_2^2(t,b(t))}=0\,.
\end{equation*}
Together with (\ref{TypMem}), assumptions \ref{c1}, \ref{con:U} and the dominated convergence 
theorem, the inner integral above converges to $0$ for each fixed $z$ as $t\rightarrow\infty$, 
and is bounded by a constant multiple of $\|f\|_{L_1}$ which is integrable with respect to the measure $e^{-\frac{|z|^2}{2}}dz$.  Hence, by the dominated convergence theorem the double integral converges to $0$ as well. This shows \eqref{eq:cme}.
\end{proof}

\begin{lemma} \label{lem:step1}
	Let $f$ be a measurable function in $L^2(\RR^d)$.  
	If \ref{c0}, \ref{c1} and \ref{con:ws1s} hold,
	\begin{equation}\label{eqn:l2lim}
		\lim_{t\to\infty} e^{-2 \beta t}\sigma^{2d}(t)\EE_{\xi[0,r]} \lt[\lt(X_t(f)-\EE_{\xi[0,r]} (X_t(f)|\mathcal{F}_{b(t))} \rt)^2\rt]=0\,.
	\end{equation}
\end{lemma}
\begin{proof}
	The estimate \eqref{est:X|Fs} in Proposition \ref{prop:PXtFt0} yields
	\begin{equation*}
		e^{-2\beta t}\sigma^{2d}(t) \EE_{\xi[0,r]} \lt[\lt(X_t(f)-\EE_{\xi[0,r]} (X_t(f)|\mathcal{F}_{b(t)}) \rt)^2\rt]
		\lesssim \frac{\sigma^{d}(t)}{\sigma^{d}_2(t,r)} e^{-\beta b(t)}\sigma^d(t) \|f\|^2_{L^2(\RR^d)} \,,
	\end{equation*}
	which converges to $0$ by \ref{c1} and \ref{con:ws1s}.
	\end{proof}

	\begin{lemma}\label{lem:step2}
	Let $f$ be a measurable function in $L^1(\RR^d)$. If \ref{c0}-\ref{con:U} and \ref{con:ws1s} hold,
	\begin{equation}\label{eqn:l3lim}
			\lim_{t\to\infty} \EE_{\xi[0,r]} \lt|e^{- \beta t}\sigma^d(t)\EE_{\xi[0,r]}(X_t(f)|\mathcal{F}_{b(t)})-e^{- \beta b(t)}X_{b(t)}(1)\mathcal Tf(U_\infty \xi_0) \rt|=0\,.
	\end{equation}
	\end{lemma}
	\begin{proof} 
		Proposition \ref{prop:XFs} yields
		\begin{align*}
			e^{-\beta t} \EE_{\xi[0,r]}(X_t(f)|\cff_{b(t)})
			=e^{-\beta b(t)}\sum_{\alpha\in\cii_{b(t)}}\EE_{\xi[0,r]}(f(\xi^\alpha_{t})|\cgg^\alpha_{b(t)})
		\end{align*}
		Hence, together with triangle inequality, we see that
		\begin{multline*}
			R(t):=\lt|e^{- \beta t}\sigma^d(t)\EE_{\xi[0,r]}(X_t(f)|\mathcal{F}_{b(t)})-e^{- \beta b(t)}X_{b(t)}(1)\mathcal Tf(U_\infty\xi_0) \rt|
			\\\le e^{-\beta b(t)}\sum_{\alpha\in\cii_{b(t)}}\lt|\sigma^d(t) \EE_{\xi[0,r]}(f(\xi^\alpha_{t})|\cgg^\alpha_{b(t)})-\mathcal Tf(U_\infty\xi_0)\rt|\,.
		\end{multline*}
		Taking expectation and applying Lemma \ref{lem:firstmom} yield
		\begin{align*}
			\EE_{\xi[0,r]}R(t)\le \EE_{\xi[0,r]} \lt|\sigma^d(t) \EE_{\xi[0,r]}(f(\xi_{t})|\cgg_{b(t)})-\mathcal Tf(U_\infty \xi_0)\rt|
		\end{align*}
		which converges to $0$ as $t\to\infty$ by Lemma \ref{lem:ergodic}.
\end{proof}
\begin{proof}[Proof of Theorem \ref{thm:wlln}] 
	We have seen in \eqref{def:F} that $e^{-\beta b(t)}X_{b(t)}(1)$ converges $\PP_{\xi[0,r]}$-almost surely to $e^{-\beta r} F$. On the other hand, the convergences in $\PP_{\xi[0,r]}$-probability of \eqref{lim:step1} and \eqref{lim:step2} are verified by Lemmas \ref{lem:step1} and \ref{lem:step2} respectively. Combining these facts yields the result.	
\end{proof}

\section{The strong law of large numbers} 
\label{sec:SLLN}
	The proof of Theorem \ref{thm:SLLN} is presented in the current section.	We follow a usual routine in showing strong laws of large numbers: first, we show almost sure convergence along a sequence of lattice times, then transfer this convergence to continuous time. Let us briefly explain the main ideas. Suppose $\{t_n\}$ is a sequence satisfying \ref{con:sl.tnmax} and $f$ is a fixed function. 
	Under conditions \ref{c0}-\ref{con:s1s}, the limits in \eqref{lim:step1} and \eqref{lim:step2} can be improved to be $\PP_{\xi[0,r]}$-a.s. along the sequence $\{t_n\}$. 
	To extend the convergence of $e^{-\beta t_n}\sigma^d(t_n) X_{t_n}(f)$ to continuous time, we use \eqref{def:Xt} to compare $X_t(f)$ with $X_{t_n}(f)$ and $X_{t_{n+1}}(f)$ for $t\in(t_n,t_{n+1}]$.

	Let us compare our approach with the literature. For branching diffusion processes,  the passage from convergence along lattice times to continuous time was employed successfully first by Asmussen and Hering in \cite{MR0420889}.  The method used in \cite{MR0420889} for showing almost sure convergence along lattice times is similar (but not identical) to ours. For the sequence $\{t_n\}=\{n \delta\}$ (with $\delta>0$)  and $f$ being indicator function of a bounded set, a main step in \cite{MR0420889} is the following almost sure limit
		\[\lim_{n\to\infty}e^{-\beta t_n}\sigma^d(t_n)\lt(X_{t_n}(f)-\EE_x(X_{t_n}(h)|\cff_{b(t_{n-1})})\rt)=0\,,\]
where $b(t)=t$ and $h$ is a principal eigenfunction of the semigroup corresponding to the mean of $X$. The passage from lattice times to continuous time was obtained using Markov's property of $X_t(f)$. This argument has been extended in later work to other situations: for more general branching diffusions in \cite{MR2641779}, for superprocesses in \cite{MR3395469,MR3352888,MR3010225}. Since eigenfunctions are used, certain conditions on the spectrum of the underlying diffusion are required. It is worth noting that these assumptions are not verified for Brownian motion.  
In our situation, the underlying process is not necessarily a diffusion on $\RR^d$, and the eigenfunctions are unknown. 
However, we succeed by adopting different choices for $t_n$, $b$ and $h$. In particular, $h=f$ and the choices for $\{t_n\}$ and $b$ are situational (see Section \ref{sec:examples} for a few examples).
In the case of Brownian motion, it is required that as $t\to\infty$, $t-b(t)$ also goes to infinity, as opposed to a finite constant as in \cite{MR0420889}.


	
	In the remaining of the current section, $X=\{X_t\}_{t\ge r}$ is a branching system starting with a typical memory $\xi[0,r]$. Almost sure convergence along lattice times is described in the following theorem.
\begin{theorem}\label{thm:lattice}
Let $\gamma>0$ and assume conditions \ref{c0}-\ref{con:s1s} are satisfied.
For every $f\in L^2(\RR^d)\cap L^1(\RR^d)$ and every sequence $\{t_n\}$ in $(r,\infty)$ satisfying $t_n\ge n^\gamma$ for all $n$ and
\begin{equation}\label{eqn:tn}
	\sum_{n=2}^\infty e^{-\beta b(t_{n-1})}\sigma^d(t_n)<\infty\,,
\end{equation}
we have with $\PP_{\xi[0,r]}$-probability one
\begin{equation}\label{eqn:SLtn}
	\lim_{n\to\infty}e^{-\beta t_n}\sigma^d(t_n)X_{t_n}(f)=e^{-\beta r} (2 \pi)^{-\frac d2}F \int_{\RR^d}\exp\lt\{-\frac12\lt|U_\infty \xi_0-\frac{y}\ell\rt|^2\rt\}f(y)dy\,.
\end{equation}		
\end{theorem}
	To prove Theorem \ref{thm:lattice}, we follow the same strategy of showing Theorem \ref{thm:wlln}. The main difference here is the convergence in probability is upgraded to almost sure convergence. More precisely, we have the following two lemmas, which are updated versions of Lemmas \ref{lem:step1} and \ref{lem:step2}.
	\begin{lemma}\label{lem:SLLN1}
		Under the hypothesis of Theorem \ref{thm:lattice}, we have
		\begin{equation}\label{est:Xnbn}
			\lim_{n\to\infty} e^{-\beta t_n}\sigma^{d}(t_n)\lt|X_{t_n}(f)-\EE_{\xi[0,r]}(X_{t_n}(f)|\cff_{b(t_{n-1})}) \rt|=0\quad \PP_{\xi[0,r]}- \mathrm{ a.s.}
		\end{equation}
	\end{lemma}
	\begin{proof}		
		The estimate in Proposition \ref{prop:PXtFt0} yields
		\begin{align*}
			e^{-2\beta t_n}\sigma^{2d}(t_n)\EE_{\xi[0,r]} \lt(\lt(X_{t_n}(f)-\EE_{\xi[0,r]} (X_{t_n}(f)|\cff_{b(t_{n-1})})\rt)^2 \rt)
			\lesssim \|f\|_{L^2(\RR^d)}^2 \frac{\sigma^d(t_n)}{\sigma^d_2(t_n,r)} e^{-\beta b(t_{n-1})}\sigma^{d}(t_n)\,.
		\end{align*}
		From Remark \ref{rem:ss}, $\lim_{t\to\infty} \frac{\sigma^d(t_n)}{\sigma^d_2(t_n,r)}=1$, it follows 
		\begin{equation*}
			\sum_{n=1}^\infty e^{-2\beta b(t_n)}\sigma^{2d}(t_n)\EE_{\xi[0,r]} \lt(\lt(X_{t_n}(f)-\EE_{\xi[0,r]}(X_{t_n}(f)|\cff_{b(t_{n-1})})\rt)^2 \rt)\lesssim \|f\|_{L^2}^2\sum_{n=1}^\infty e^{-\beta b(t_{n-1})}\sigma^d(t_n)\,,
		\end{equation*}
		which is finite due to \eqref{eqn:tn}. Hence, an application of Borel-Cantelli's lemma yields \eqref{est:Xnbn}. 
	\end{proof}
	\begin{lemma}\label{lem:SLLN2} Under the hypothesis of Theorem \ref{thm:lattice}, we have
		\begin{equation}
			\lim_{n\to\infty}\lt|e^{-\beta t_n}\sigma^d(t_n)\EE_{\xi[0,r]}(X_{t_n}(f)|\cff_{b(t_{n-1})})-e^{-\beta b(t_{n-1})}X_{b(t_{n-1})}(1) \mathcal Tf(U_\infty \xi_0)\rt|=0\quad\PP_{\xi[0,r]}-\mathrm{a.s.}
		\end{equation}
		where we recall $\mathcal{T}f$ is defined in \eqref{def:Tf}.
	\end{lemma}
	\begin{proof}
		For each $\alpha\in\cii^l$ and $n$ sufficiently large, we put
		\begin{equation*}
			\eta_n^\alpha
			=\int_r^{b(t_{n-1})}K(t_n,u)dW^\alpha_u\,.
		\end{equation*}
		Each $\eta_n^\alpha$ is centered normal random variable with variance $\sigma_n^2=\int_r^{b(t_{n-1})}|K(t_n,u)|^2du$.
		Note that if $\alpha$ is alive at time $b(t_{n-1})$ then $\eta_n^\alpha$ is $\cff_{b(t_{n-1})}$-measurable. 
		Let $a_n=2\sqrt{\log n}$,  $A_n=\{z\in\RR^d:|z|<a_n \sigma_n \}$. $A_n^c$ denotes the complement of $A_n$. 
		As in the proof of Lemma \ref{lem:step2}, applying Proposition \ref{prop:XFs} we have
		\begin{align*}
			&\lt|e^{-\beta t_n}\sigma^d(t_n) \EE_{\xi[0,r]}(X_{t_n}(f)|\cff_{b(t_{n-1})}) -e^{-\beta b(t_{n-1})}X_{b(t_{n-1})}(1) \mathcal T f(U_\infty x)\rt|
			\\&\le e^{-\beta b(t_{n-1})}\sum_{\alpha\in\cii_{b(t_{n-1})}}\lt|\sigma^d(t_n) \EE_{\xi[0,r]}(f(\xi^\alpha_{t_n})|\cgg_{b(t_{n-1})}^\alpha)- \mathcal T f(U_\infty x)\rt|
			\\&=J_1(n)+J_2(n)\,,
		\end{align*}
		where
		\begin{equation*}
			J_1(n)=e^{-\beta b(t_{n-1})}\sum_{\alpha\in\cii_{b(t_{n-1})}}\lt|\sigma^d(t_n) \EE_{\xi[0,r]}(f(\xi^\alpha_{t_n})|\cgg_{b(t_{n-1})}^\alpha)- \mathcal T f(U_\infty x)\rt|1_{A_n}(\eta_{n}^\alpha)\,,
		\end{equation*}
		\begin{equation*}
			J_2(n) =e^{-\beta b(t_{n-1})}\sum_{\alpha\in\cii_{b(t_{n-1})}}\lt|\sigma^d(t_n) \EE_{\xi[0,r]}(f(\xi^\alpha_{t_n})|\cgg_{b(t_{n-1})}^\alpha)- \mathcal T f(U_\infty x)\rt|1_{A_n^c}(\eta_{n}^\alpha)\,.
		\end{equation*}
		We will show below that $J_1(n)$ and $J_2(n)$ converge to 0 almost surely. Let us consider $J_1$ first. Let $M$ be a fixed positive number. We put
	\begin{equation*}
		s_n= \sup_{|y|\le M,|z|\le a_n \sigma_n}g_n(y,z )\,,
	\end{equation*}
	where 
	\begin{multline*}
	g_n(y,z)= \lt|\frac{\sigma^d(t_n)}{\sigma_2^d(t_n,b(t_{n-1}))} \exp\lt\{-\frac{|\int_0^rK(t_n,u)dW_u+ z +U_{t_n} \xi_0-y|^2}{2 \sigma_2^2(t_n,b(t_{n-1}))} \rt\}\rt.
		\\\lt. -\exp\lt\{-\frac12\lt|U_\infty \xi_0-\frac{y}\ell\rt|^2\rt\} \rt|\,.
\end{multline*}
By triangle inequality
\begin{multline*}
	\sup_{|y|\le M,|z|\le a_n \sigma_n} \lt|\frac{|\int_0^rK(t_n,u)dW_u+z  +U_{t_n} \xi_0-y|}{ \sigma_2(t_n,b(t_{n-1}))}-\lt|U_\infty \xi_0-\frac{y}\ell\rt| \rt|
	\\\le \frac{|\int_0^rK(t_n,u)dW_u|}{\sigma_2(t_n, b(t_{n-1}))}+ \frac{a_n \sigma_n}{\sigma_2(t_n, b(t_{n-1}))}
	+\lt|\frac{U_{t_n}\xi_0}{\sigma_2(t_n, b(t_{n-1}))}-U_\infty \xi_0\rt|
	+M\lt|\frac1{\sigma_2(t_n, b(t_{n-1}))}-\frac1\ell \rt|\,,
\end{multline*}
which converges to $0$ as $n\to\infty$ by \ref{c1}-\ref{con:s1s}, Remark \ref{rem:ss} and the typical memory assumption. 
(For clarity, the second term goes to $0$ by (C3) since $t_n\ge n^\gamma$
and\\
\(\displaystyle \frac{a_n \sigma_n}{\sigma_2(t_n, b(t_{n-1}))}=2\sqrt{\frac{\log(n)}{\log(t_n)}}
\sqrt{\log(t_n)}\frac{|\sigma_1^2(t_n, b(t_{n-1}))-\sigma_1^2(t_n, r)|^\frac12/\sigma(t_n)}{|\sigma^2(t_n)-\sigma_1^2(t_n,b(t_{n-1}))|^\frac12/\sigma(t_n)}\rightarrow \frac{0-0}{1-0}=0
\).)
It follows that $s_n$ also converges to 0 as $n\to\infty$.
		Using explicit densities of normal random variables, we have
		\begin{align*}
			\lt|\sigma^d(t_n) \EE_{\xi[0,r]}(f(\xi^\alpha_{t_n})|\cgg_{b(t_{n-1})}^\alpha)- \mathcal T f(U_\infty \xi_0)\rt|1_{A_n}(\eta_{n}^\alpha)
			&\lesssim\int_{\RR^d}g_n(y,\eta_{n}^\alpha) f(y)dy1_{A_n}(\eta_{n}^\alpha)
			\\&\lesssim s_n\int_{|y|\le M} |f(y)|dy+\int_{|y|>M}|f(y)|dy\,.
		\end{align*}
		It follows that
		\begin{equation*}
			J_1(n)\lesssim \lt(s_n\int_{|y|\le M} |f(y)|dy+\int_{|y|>M}|f(y)|dy\rt) e^{-\beta b(t_{n-1})}X_{b(t_{n-1})}(1)\,.
		\end{equation*}
		Sending $n$ to infinity and using \eqref{def:F}, we obtain
		\begin{equation*}
			\limsup_{n\to\infty}J_1(n)\lesssim F\int_{|y|>M}|f(y)|dy\quad\PP_{\xi[0,r]}-\mathrm{a.s.}
		\end{equation*}
		Sending $M$ to infinity implies that $J_1(n)$ converges to 0 almost surely. 

		For $J_2$, we note that $\lt|\sigma^d(t_n) \EE_{\xi[0,r]}(f(\xi^\alpha_{t_n})|\cgg_{b(t_{n-1})}^\alpha)- \mathcal T f(U_\infty \xi_0)\rt|$ is bounded by a constant times $\|f\|_{L^1(\RR^d)}$. Thus,
		\begin{equation*}
			J_2(n)\lesssim \|f\|_{L^1(\RR^d)}e^{-\beta b(t_{n-1})}\sum_{\alpha\in\cii_{b(t_{n-1})}}1_{A_n^c}(\eta_{n}^\alpha)\,.
		\end{equation*}
		Now applying Lemma \ref{lem:firstmom} yields
		\begin{equation*}
			\EE_{\xi[0,r]}J_2(n)\lesssim \|f\|_{L^1(\RR^d)}\PP_{\xi[0,r]}(|\eta_{n}|>a_n \sigma_n)\,.
		\end{equation*}
		In addition, by a standard tail estimate for normal random variables
		\begin{equation*}
			\PP_{\xi[0,r]}(|\eta_{n}|>a_n \sigma_n)\le 2 \frac{e^{-a_n^2/2}}{a_n}=\frac1{n^2 \sqrt{\ln n}}\,.
		\end{equation*}
		Altogether, we obtain
		\begin{equation*}
			\sum_{n=2}^\infty\EE_{\xi[0,r]}J_2(n) \lesssim\|f\|_{L^1(\RR^d)} \sum_{n=2}^\infty \frac1{n^2 \sqrt{\ln n}}\,,
		\end{equation*}
		which is finite. By Borel-Cantelli lemma, this implies $J_2(n)$ converges to 0 almost surely. 
	\end{proof}
	\begin{proof}[Proof of Theorem \ref{thm:lattice}]
		It is evident from Lemmas \ref{lem:SLLN1}, \ref{lem:SLLN2} and \eqref{def:F}.
	\end{proof}
	We state the following result from \cite[Lemma 7]{MR3131303}, which is based upon Theorem 6 in \cite{MR2673979}.
\begin{lemma}
\label{countable}Let  $E$ be a topological space with a countable base and $B(E)$ be the space of bounded Borel measurable functions on $E$. Suppose that $ \left\{\nu _{t}\right\}\cup \left\{\nu \right\}$ are (possibly non-finite) Borel measures; $ f\in B\left(E\right)$ satisfies $ 0<\nu \left(f\right)<\infty$; $\mathcal M\subset B(E)$ strongly separates points, is countable and is closed under multiplication; and
 \[\nu _{t}\left(gf\right)\rightarrow \nu \left(gf\right)\]for all $ g\in \mathcal M\cup \{1\}$. Then,
 \[\nu _{t}\left(gf\right)\rightarrow \nu \left(gf\right)\]for all bounded continuous functions $g$ on $E$.
\end{lemma}

	We are now ready to prove Theorem \ref{thm:SLLN}.
	\begin{proof}[Proof of Theorem \ref{thm:SLLN}]
		Here is an outline of our strategy: in the first two steps, we show the strong law for a fixed but arbitrary  function of a specific form. In the final step, we obtain the full statement of Theorem \ref{thm:SLLN} by an application of Lemma \ref{countable}.

	Let $\{t_n\}$ be the sequence in \ref{con:sl.tnmax}. Let $D$ be a measurable set of $\RR^d$ whose boundary has measure 0. If $\ell$ is infinite, we put $f_a(x)=e^{-\frac{|x|}{a}}$ for each $a\in(0,\infty)$. If $\ell$ is finite, $f_a\equiv1$ for all $a>0$. We impose that $1_D f_a$ belongs to $L^1(\RR^d)\cap L^2(\RR^d)$. In particular, if $\ell$ is finite, $D$ must have finite measure, otherwise, $D$ can have infinite measure.  As our first goal, we will show that
	\begin{equation}\label{tmp:sl1D}
		\lim_{t\to\infty}e^{-\beta t}\sigma^d(t)X_t(1_D f_{a})=e^{-\beta r}F \mathcal T (1_D f_a)(U_\infty \xi_0)\quad\PP_{\xi[0,r]}- \mathrm{a.s.}\,,
	\end{equation}
	where $\mathcal T$ is defined in \eqref{def:Tf}.
	We note that Theorem \ref{thm:lattice} already implies the $\PP_{\xi[0,r]}$-a.s. convergence along the sequence $\{t_n\}$.

		\textbf{Step 1:} Let us show the lower bound of \eqref{tmp:sl1D}. Let $\epsilon$ be a fixed positive number. We denote $D_ \epsilon=\{z\in D:\mathrm{dist}(z,\partial D)>\epsilon\}$ and
		\begin{equation}\label{def:Analpha}
			A^\alpha_n=\{\sup_{u,v\in[t_n,t_{n+1}]}|\xi^\alpha_u- \xi	^\alpha_{v}|\ge \epsilon\} \,.
		\end{equation}
		From \eqref{def:Xt}, under $\PP_{\xi[0,r]}$, we can write
		\begin{equation}\label{tmp:XtD}
			X_t(1_Df_a)=\sum_{\alpha\in\cii^l}1_{[b^\alpha,d_\alpha)}(t)1_D(\xi_t^\alpha) f_a(\xi^\alpha_t)	.
		\end{equation} 
		For every $t\in[t_n,t_{n+1}]$, we have
		\begin{equation*}
			1_{[b^\alpha,d_\alpha)}(t)\ge 1_{[b^\alpha,d_\alpha)}(t_n)-1_{[b^\alpha,d_\alpha)}(t_n)1_{[t_n,t_{n+1}]}(d_\alpha)
		\end{equation*}
and
\begin{equation*}
	e^{\frac{\epsilon}{a}}1_D(\xi^\alpha_t)f_a(\xi^\alpha_t) \ge  1_{D_ \epsilon}(\xi^\alpha_{t_n})f_a(\xi^\alpha_{t_n}) -1_{A^\alpha_n}1_{D_ \epsilon}(\xi^\alpha_{t_n})f_a(\xi^\alpha_{t_n})\,.
\end{equation*}
We note that right-hand sides of the previous estimates are non-negative. Substitute these estimates into \eqref{tmp:XtD}, we get
		\begin{align*}
			e^{\frac{\epsilon}{a}}X_t(1_D f_a)\ge X_{t_n}(1_{D_ \epsilon}f_a)&- \sum_{\alpha\in\cii^l} 1_{[b^\alpha,d_\alpha)}(t_n)1_{A_n^\alpha}1_{D_ \epsilon}(\xi^\alpha_{t_n})f_a(\xi^\alpha_{t_n})
			\\&-\sum_{\alpha\in\cii^l}1_{D_ \epsilon}(\xi^\alpha_{t_n})f_a(\xi^\alpha_{t_n})1_{[b^\alpha,d_\alpha)}(t_n)1_{[t_n,t_{n+1}]}(d_\alpha)\,.
		\end{align*}
		Bounding $f_a$ and indicator functions by 1 in the last two terms, it follows that for every $t\in[t_n,t_{n+1}]$,
		\begin{equation}\label{tmp:lowX}
			e^{\frac{\epsilon}{a}}X_t(1_Df_a)\ge X_{t_n}(1_{D_ \epsilon}f_a)-Y_n-Z_n\,,
		\end{equation}
		where
		\begin{equation*}
			Y_n=\sum_{\alpha\in\cii^l} 1_{[b^\alpha,d_\alpha)}(t_n)1_{A_n^\alpha}=\sum_{\alpha\in\cii_{t_n}} 1_{A_n^\alpha}
		\end{equation*}
		and
		\begin{equation*}
			Z_n=\sum_{\alpha\in\cii}1_{D_ \epsilon}(\xi^\alpha_{t_n})f_a(\xi^\alpha_{t_n})1_{[b^\alpha,d_\alpha)}(t_n)1_{[t_n,t_{n+1}]}(d_\alpha)=\sum_{\alpha\in\cii_{t_n} }1_{D_ \epsilon}(\xi^\alpha_{t_n})f_a(\xi^\alpha_{t_n})1_{[t_n,t_{n+1}]}(d_\alpha)\,.
		\end{equation*}
		We show now that $e^{-\beta t_n}\sigma^d(t_n) Y_n$ and $e^{-\beta t_n}\sigma^d(t_n)Z_n$ converge to 0, $\PP_{\xi[0,r]}$-almost surely. From Lemma \ref{lem:firstmom}, we have
		\begin{align*}
			\EE_{\xi[0,r]} Y_n 
			=e^{\beta (t_n-r)}\PP_{\xi[0,r]}\lt(\sup_{u,v\in[t_n,t_{n+1}]}|\xi_u- \xi_v|\ge \epsilon\rt)\,.
		\end{align*}
		Hence, 
		\begin{align*}
			\sum_{n=1}^\infty e^{-\beta t_n}\sigma^d(t_n) \EE_{\xi[0,r]} Y_n\lesssim\sum_{n=1}^\infty \sigma^d(t_n)\PP_{\xi[0,r]} \lt(\sup_{u,v\in[t_n,t_{n+1}]}|\xi_u-\xi_v|\ge \epsilon\rt)\,,
		\end{align*}
		which is finite by condition \ref{con:sl.tnmax}. Therefore, by Borel-Cantelli lemma, we get
		\begin{equation*}
			\lim_{n\to\infty}e^{-\beta t_n}\sigma^d(t_n) Y_n=0\,.
		\end{equation*}

		The term $Z_n$ is a little more delicate. 
		We note that,
		\begin{align*}
		 	\EE_{\xi[0,r]}(Z_n|\cff_{t_n})=\sum_{\alpha\in\cii_{t_n}}1_{D_ \epsilon}(\xi_{t_n}^\alpha)f_a(\xi^\alpha_{t_n})\EE_{\xi[0,r]}(1_{[t_n,t_{n+1}]}(d_ \alpha)|\cff_{t_n})
		\end{align*} 
		and for each $\alpha\in\cii_{t_n}$, we have
		\begin{align*}
			\EE_{\xi[0,r]}(1_{[t_n,t_{n+1}]}(d_ \alpha)|\cff_{t_n})
			&=\PP(t_n\le L^\alpha +b^\alpha\le t_{n+1}|L^\alpha+b^\alpha\ge t_n, b^\alpha)=1-e^{-V(t_{n+1}-t_n)}\,.
		\end{align*}
		To simplify our notation, we put $c_n=1-e^{-V(t_{n+1}-t_n)}$. It follows that
		\begin{align}\label{tmp:znfn}
		 	\EE_{\xi[0,r]}(Z_n|\cff_{t_n})
		 	=c_n X_{t_n}(1_{D_ \epsilon}f_a)
		\end{align} 
		and hence,
		\begin{align*}
			&|Z_n-\EE_{\xi[0,r]}( Z_n|\cff_{t_n})|^2
			\\&=\lt(\sum_{\alpha\in\cii_{t_n}}1_{D_ \epsilon}(\xi^\alpha_{t_n})f_a(\xi^\alpha_{t_n})
			(1_{[t_n,t_{n+1}]}(d_ \alpha)-c_n) \rt)^2
			\\&=\sum_{\alpha\in\cii_{t_n}}1_{D_ \epsilon}(\xi^\alpha_{t_n})f_a(\xi^\alpha_{t_n})
			(1_{[t_n,t_{n+1}]}(d_ \alpha)-c_n)^2
			\\&\quad+\sum_{\alpha,\alpha'\in\cii_{t_n};\alpha\neq \alpha'}1_{D_ \epsilon}(\xi^\alpha_{t_n})f_a(\xi^\alpha_{t_n})1_{D_ \epsilon}(\xi^{\alpha'}_{t_n})f_a(\xi^{\alpha'}_{t_n})(1_{[t_n,t_{n+1}]}(d_ \alpha)-c_n)(1_{[t_n,t_{n+1}]}(d_ {\alpha'})-c_n)\,.
		\end{align*}
		Conditioning on $\cff_{t_n}$, $d_ \alpha$ and $d_ {\alpha'}$ are independent if $\alpha,\alpha'\in\cii_{t_n}$ and $\alpha\neq {\alpha'}$. Hence the second sum above vanishes upon taking expectation. It follows that
		\begin{equation*}
			\EE_{\xi[0,r]} 	|Z_n-\EE_{\xi[0,r]}( Z_n|\cff_{t_n})|^2\le  \EE_{\xi[0,r]}\sum_{\alpha\in\cii_{t_n}}1_{D_ \epsilon}(\xi^\alpha_{t_n})f_a(\xi^\alpha_{t_n})(1_{[t_n,t_{n+1}]}(d_ \alpha)-c_n)^2
			\le 4\EE_{\xi[0,r]}X_{t_n}(1_{D_ \epsilon}f_a)\,.
		\end{equation*}
		Applying Lemma \ref{lem:firstmom}, we have
		\begin{align*}
			\EE_{\xi[0,r]} 	|Z_n-\EE_{\xi[0,r]}( Z_n|\cff_{t_n})|^2\lesssim e^{\beta t_n}\EE_{\xi[0,r]}1_{D_ \epsilon}(\xi_{t_n})f_a(\xi_{t_n})\lesssim e^{\beta t_n}\sigma_2^{-d}(t_n,r)\,.
		\end{align*}
		So
		\begin{equation*}
		 	\sum_{n=1}^\infty e^{-2\beta t_n}\sigma^{2d}(t_n)\EE_{\xi[0,r]} 	|Z_n-\EE_{\xi[0,r]}( Z_n|\cff_{t_n})|^2\lesssim \sum_{n=1}^\infty\frac{\sigma^d(t_n)}{\sigma_2^d(t_n,r)}  e^{-\beta t_n}\sigma^d(t_n)\,,
		 \end{equation*} 
		which is  finite by \ref{con:s1s} and \ref{con:sl.tnmax}.
		Applying Borel-Cantelli lemma, we get
		\begin{align*}
			\lim_{n\to\infty}e^{-\beta t_n}\sigma^d(t_n)Z_n=\lim_{n\to\infty}e^{-\beta t_n}\sigma^d(t_n)\EE_{\xi[0,r]}(Z_n|\cff_{t_n})\,.
		\end{align*}
		On the other hand, $t_{n+1}-t_n\to0$ as $n\to\infty$, by \eqref{tmp:znfn} and Theorem \ref{thm:lattice}, it follows that
		\begin{align*}
			\lim_{n\to\infty}e^{-\beta t_n}\sigma^d(t_n) \EE_{\xi[0,r]}(Z_n|\cff_{t_n})
			=\lim_{n\to\infty}e^{-\beta t_n}\sigma^d(t_n)X_{t_n}(1_{D_ \epsilon}f_a)(1-e^{-V(t_{n+1}-t_n)})=0
		\end{align*}
		with probability one.
		Upon combining these limit identities together, we obtain
		\begin{equation*}
			\lim_{n\to\infty} e^{-\beta t_n}\sigma^d(t_n)Z_n=0\quad\PP_{\xi[0,r]} -\mathrm{a.s.}
		\end{equation*}
		Together with Theorem \ref{thm:lattice}, condition \ref{con:sl.tnmax} and the estimate \eqref{tmp:lowX}, we get
		\begin{equation*}
			e^{\frac{\epsilon}{a}} \liminf_{t\to\infty}e^{-\beta t}\sigma^d(t)X_t(1_Df_a)\ge e^{-\beta r}F \mathcal T(1_{D_ \epsilon}f_a)(U_\infty \xi_0)\quad \PP_{\xi[0,r]} - \mathrm{a.s.}
		\end{equation*}
		We now let $\epsilon\downarrow0$ to achieve the lower bound of \eqref{tmp:sl1D}.

		\textbf{Step 2:} The upper bound of \eqref{tmp:sl1D} is more involved. Using the inequality
	\begin{equation*}
	 1_{[b^\alpha,d_ \alpha)}(t)\le 1_{[b^\alpha,d_ \alpha)}(t_{n+1})+1_{[b^\alpha,d_ \alpha)}(t)1_{[t_n,t_{n+1}]}(d_ \alpha)\,\ \forall t\in[t_n,t_{n+1}],
	\end{equation*} 
	we see from \eqref{tmp:XtD} that
	\begin{align}\label{tmp:upXIII}
			X_t(1_Df_a)&\le \EE_{\xi[0,r]} \lt(\sum_{\alpha\in\cii}1_{[b^\alpha,d_ \alpha)}(t_{n+1})1_{D}(\xi^\alpha_t)f_a(\xi^\alpha_t)\Big|\cff_t\rt)
			\nonumber\\&\quad+\EE_{\xi[0,r]} \lt(\sum_{\alpha\in\cii} 1_{[b^\alpha,d_ \alpha)}(t)1_{[t_n,t_{n+1}]}(d_ \alpha)1_D(\xi^\alpha_t)f_a(\xi^\alpha_t) \Big|\cff_t\rt)\,.
		\end{align}
		We denote by $I_t$ and $II_t$ the first and second conditional expectation in the right-hand side above.
		For each $\epsilon>0$, let $D^\epsilon=\{z\in \RR^d:\mathrm{dist}(z,\partial D)<\epsilon\}$. Let us recall that $A_n^\alpha$ is defined in \eqref{def:Analpha}. To handle $I_t$, we employ the inequality
		\begin{equation*}
			1_{D}(\xi^\alpha_t)f_a(\xi^\alpha_t)\le e^{\frac{\epsilon}{a}} 1_{D^\epsilon}(\xi^\alpha_{t_{n+1}})f_a(\xi^\alpha_{t_{n+1}}) +1_{A_n^\alpha}
		\end{equation*}
		and arrive at
		\begin{align}\label{tmp:I}
			I_t&\le e^{\frac{\epsilon}{a}}\EE_{\xi[0,r]} (X_{t_{n+1}}(1_{D^\epsilon}f_a)|\cff_t)+ \EE_{\xi[0,r]} \lt(\sum_{\alpha\in\cii}1_{[b^\alpha,d_ \alpha)}(t_{n+1})1_{A^\alpha_n}\Big|\cff_t\rt )
			\nonumber\\&\le e^{\frac{\epsilon}{a}} X^*_n+Y^*_n\,,
		\end{align}
		where
		\begin{equation*}
			X^*_n=\sup_{s\in[t_n,t_{n+1}]}\EE_{\xi[0,r]} (X_{t_{n+1}}(1_{D^\epsilon}f_a)|\cff_s)
		\end{equation*}
		and
		\begin{equation*}
			Y^*_n=\sup_{s\in[t_n,t_{n+1}]}\EE_{\xi[0,r]} \lt(\sum_{\alpha\in\cii}1_{[b^\alpha,d_ \alpha)}(t_{n+1})1_{A^\alpha_n}\Big|\cff_s\rt ).
		\end{equation*}
		For $II_t$, we note that on the event $\{b^\alpha\le t<d_\alpha\}$
		\begin{align*}
			\EE_{\xi[0,r]} (1_{[t_n,t_{n+1}]}(d_ \alpha)|\cff_t )&=\PP(t_n\le L^\alpha+b^\alpha\le t_{n+1}|L^\alpha+b^\alpha> t, b^\alpha)
			\\&\le1-e^{-V(t_{n+1}-t_n)}\,.
		\end{align*}
		Hence,
		\begin{align}\label{tmp:II}
			II_t&=\sum_{\alpha\in\cii}1_{[b^\alpha,d_ \alpha)}(t)1_D(\xi^\alpha_t)f_a(\xi^\alpha_t) \EE_{\xi[0,r]} (1_{[t_n,t_{n+1}]}(d_ \alpha)|\cff_t)
			\nonumber\\&\le(1-e^{- V(t_{n+1}-t_n)})X_t(1_Df_a)\,.
		\end{align}
		Combining \eqref{tmp:upXIII} with \eqref{tmp:I} and \eqref{tmp:II} yields
		\begin{align}\label{tmp:XXY}
			X_t(1_D)\le e^{V(t_{n+1}-t_n)} (e^{\frac{\epsilon}{a}}X^*_n+ Y^*_n)\,.
		\end{align}
		We observe that the process $[t_n,t_{n+1}]\ni s\mapsto \EE_{\xi[0,r]} (X_{t_{n+1}}(1_{D^\epsilon}f_a)|\cff_s)$ is a martingale and has a right-continuous modification (recall that $\{\cff_t\}$ is a right-continuous filtration). In addition, since $b(t_n)<t_n$, the random variable $K_n:=\EE_{\xi[0,r]} (X_{t_{n+1}}(1_{D^ \epsilon}f_a) |\cff_{b(t_{n})})$ is $\cff_{t_n}$-measurable. Note that by Jensen's inequality,
\begin{equation*}
 	\EE_{\xi[0,r]}K_n^2\le \EE_{\xi[0,r]}\lt[(X_{t_{n+1}}(1_{D^ \epsilon}f_a))^2\rt]\,.
\end{equation*} 
Using Doob's maximal inequality (see \cite{MR1725357}*{Thm. 1.7, Chap. II} or \cite{MR1796539}*{p. 177}) and Proposition \ref{prop:PXtFt0}, we see that
\begin{align*}
	\EE_{\xi[0,r]} (X^*_n-K_n)^2
	&\le \EE_{\xi[0,r]} \lt[\sup_{s\in[t_n,t_{n+1}]}\EE_{\xi[0,r]} \lt(|X_{t_{n+1}}(1_{D^\epsilon}f_a)- K_n|\big|\cff_s\rt)^2 \rt]
	\\&\le 4\EE_{\xi[0,r]}  (X_{t_{n+1}}(1_{D^\epsilon}f_a)-K_n)^2
			\\&\lesssim \EE_{\xi[0,r]}  \lt(X_{t_{n+1}}(1_{D^\epsilon}f_a)-\EE_{\xi[0,r]}\lt(X_{t_{n+1}}(1_{D^\epsilon}f_a)|\cff_{b(t_n)}\rt) \rt)^2
			\\&\lesssim e^{2 \beta t_{n+1}-\beta b(t_{n})}\sigma^{-d}_2(t_{n+1},r)\,.
		\end{align*}
		Thus, 
		\begin{align*}
			\sum_{n=1}^\infty e^{-2\beta t_{n+1}}\sigma^{2d}(t_{n+1}) \EE_{\xi[0,r]} (X^*_n-K_n )^2\lesssim \sum_{n=1}^\infty \frac{\sigma^d(t_{n+1})}{\sigma_2^2(t_{n+1},r)} e^{-\beta b(t_{n})}\sigma^d(t_{n+1})\,,
		\end{align*}
		which is finite due \ref{con:s1s} and \ref{con:sl.tnmax}. By Borel-Cantelli lemma and Lemma \ref{lem:SLLN2}, we conclude that
		\begin{align*}
			\lim_{n\to\infty}e^{-\beta t_{n+1}}\sigma^d(t_{n+1}) X^*_n
			&=\lim_{n\to\infty}e^{-\beta t_{n+1}}\sigma^d(t_{n+1}) \EE_{\xi[0,r]} (X_{t_{n+1}}(1_{D^ \epsilon}f_a) |\cff_{b(t_{n})})
			\\&=e^{-\beta r}F \mathcal T(1_{D^\epsilon}f_a) (U_\infty \xi_0)  \quad \PP_{\xi[0,r]} - \mathrm{a.s.}
		\end{align*}
$Y^*_n$ can be treated similarly.
Using the right continuity of $\{\cff_t\}$ and Doob's maximal inequality again (see \cite{MR1725357}*{Thm. 1.7, Chap. II}
) as well as Lemma \ref{lem:firstmom}, we see that for every fixed $\rho>0$,
\begin{align*}
	\EE_{\xi[0,r]} (e^{-\beta t_{n+1}}\sigma^d(t_{n+1}) Y_n^*\ge \rho)
	&\lesssim \rho^{-1}e^{-\beta t_{n+1}}\sigma^d(t_{n+1}) \EE_{\xi[0,r]}\sum_{\alpha\in\cii}1_{[b^\alpha,d_ \alpha)}(t_{n+1})1_{A^\alpha_n} 
	\\&\lesssim \rho^{-1} \sigma^d(t_{n+1})\PP_{\xi[0,r]}\lt(\sup_{u,v\in[t_n,t_{n+1}]}|\xi_u- \xi_v|\ge \epsilon \rt).
\end{align*}
Thanks to \ref{con:sl.tnmax}, we can apply Borel-Cantelli lemma to conclude that
\begin{equation*}
	\lim_{n\to\infty}e^{-\beta t_{n+1}}\sigma^d(t_{n+1})Y^*_n=0\quad\PP_{\xi[0,r]} -\mathrm{a.s.}
\end{equation*}
		Now these two newly established limits together with \eqref{tmp:XXY} and \eqref{limtn} imply
		\begin{equation*}
			\limsup_{t\to\infty}e^{-\beta t}\sigma^d(t) X_t(1_Df_a)\le e^{\frac{\epsilon}{a}}e^{-\beta r}F \mathcal T(1_{D^\epsilon}f_a)(U_\infty \xi_0)\,,
		\end{equation*}
		which upon sending $\epsilon$ to 0 yields the upper bound of \eqref{tmp:sl1D}. 

		\textbf{Step 3:} Depending on the finiteness of $\ell$, we have two separate cases. We only consider here the case when $\ell$ is infinite. The remaining case is treated similarly and omitted.  Note that $C_c(\mathbb R^d)$ is an algebra in $B(\RR^d)$ that strongly separates points. Hence, by \cite{MR2673979}*{Lemma 2}, there is a countable subcollection $\mathcal M\subset C_c(\RR^d)$ that strongly separates points and is closed under multiplication.
		 We show below that with $\PP_{\xi[0,r]}$-probability one
		\begin{equation}\label{eqn:step3}
			\lim_{t\to\infty}e^{-\beta t}\sigma^d(t)X_t(gf_n)=e^{-\beta r}F\mathcal T(gf_n)(U_\infty \xi_0)\,.
		\end{equation}
		for all $g\in \mathcal M\cup{1}$ and $n\in\NN$.
		Fix $n\in\NN$, for each $g\in \mathcal M$, there exist two sequences of step functions $\{\overline g_m\}$ and $\{\underline {g}_m\}$ which converge to $g$ pointwise from below and above. In particular, we have
		\begin{equation*}
			\underline{g}_m(x)f_n(x)\le g(x) f_n(x)\le \overline g_m(x)f_n(x)
		\end{equation*}
		for all $x\in\RR^d$.  It follows from Steps 1 and 2 that with $\PP_{\xi[0,r]}$-probability one,
		\begin{align*}
			\limsup_{t\to\infty}e^{-\beta t}\sigma^d(t)X_t(g f_n)\le e^{-\beta r}F \mathcal T(\overline g_m f_n)(U_\infty \xi_0)
		\end{align*}
		and 
		\begin{align*}
			\liminf_{t\to\infty}e^{-\beta t}\sigma^d(t)X_t(gf_n)\ge e^{-\beta r}F \mathcal T(\underline g_mf_n)(U_\infty \xi_0)
		\end{align*}
		for all $m\in\NN$. 
		Now let $m\to\infty$ we obtain \eqref{eqn:step3} for fixed $g$ and $f_n$. If $g\equiv 1$, \eqref{eqn:step3} follows directly from Steps 1 and 2.
		Since $\mathcal M$ and $\{f_n\}$ are countable, one can find an event $\Omega^*$ such that $\PP_{\xi[0,r]}(\Omega^*)=1$ and on $\Omega^*$, \eqref{eqn:step3} holds for all $g\in \mathcal M\cup\{1\}$ and all $n\in\NN$.

		Hereafter, we fix a realization in $\Omega^*$. Let $g$ be a continuous function on $\RR^d$ and $n\in\NN$ be such that $\sup_{x\in\RR^d}e^{\frac{|x|}n} |g(x)|$ is finite.  On the event that $F>0$, we can take $f\equiv  f_n$  in Lemma \ref{countable} to get the convergence of $e^{-\beta t}\sigma^d(t)X_t(g)$ to the desired limit. On the event that $F=0$, we note that $|X_t(g)|\le\sup_{x\in\RR^d}e^{\frac{|x|}n} |g(x)|X_t(f_n)$ and hence, it follows from \eqref{eqn:step3} that $e^{-\beta t}\sigma^d(t)X_t(g)$ converges to 0. Hence, \eqref{eqn:slln} is proved and the proof of Theorem \ref{thm:SLLN} is complete.
	\end{proof}

\section{Examples} 
\label{sec:examples}
	We present in the current section two classes of spatial motions, each of which exhibits different long term dynamic from the other. To be precise, we denote by $\{B^H_t,t\in\RR\}$ a two-sided, normalized fractional Brownian motion with Hurst parameter $H\in(0,1)$. In particular, $B^H$ satisfies the following properties.
	\begin{enumerate}[(i)]
		\item $B^H_0=0$ and $\EE B^H_t=0$ for all $t\in\RR$.
		\item $B^H$ has homogeneous increments, i.e., $B^H_{t+s}-B^H_s$ has the same law of $B^H_t$ for $s,t\in\RR$.
		\item $B^H$ is a Gaussian process and $\EE |B^H_t|^2=t^{2H}$ for all $t\in\RR$.
		\item $B^H$ has continuous trajectories.
	\end{enumerate}
	The law of $B^H$ and its corresponding expectation are denoted by $\PP$ and $\EE$ respectively.

	When $H=\frac12$, the fractional Brownian motion $B^{\frac12}$ coincides with the standard Brownian motion. Fractional Brownian motions with $H>\frac12$ have long range dependence, that is
	\begin{equation*}
	 	\sum_{n=1}^\infty \EE [B^H_t(B^H_{(n+1)t}-B^H_{nt})]=\infty\,.
	\end{equation*} 
	Unless $H= \frac12$, the fractional Brownian motion is neither a Markov process nor a semimartingale (cf. \cite{BiaginiHuOksendalZhang08}). It is, however, shown in \cite{NorrosValkeilaVirtamo99} that $B^H$ is of the Volterra form \eqref{rep:volterra}.

	We briefly recall a construction of stochastic integration $\int_\RR f(s) dB_s^H$ where $f$ is deterministic and belongs to some suitable function spaces. For every function $f$ on $\RR$, we denote by $\hat f$ its Fourier transform with the following normalization
	\begin{equation*}
		\hat f(x)=\int_{\RR}e^{\sqrt{-1}s x}f(s)ds\,,
	\end{equation*}
	where $\sqrt{-1}$ is the imaginary unit.
	If $f$ is an elementary (or step) function given by
	\begin{equation*}
		f(u)=\sum_{k=1}^nf_k 1_{[u_k,u_{k+1})}(u),\quad u\in\RR\,,
	\end{equation*}
	then we define $\int_\RR f(s)dB^H_s$ as the Riemann sum
	\begin{equation*}
		\sum_{k=1}^n f_k (B^H_{u_k}-B^H_{u_{k+1}})\,.
	\end{equation*}
	It is shown in \cite[page 257] {MR1790083} that for every $H\in(0,1)$ and elementary functions $f,g$,
	\begin{equation}\label{id:iso}
		\EE\lt[\int_\RR f(s)dB_s^H\cdot \int_\RR g(s)dB_s^H\rt]=c_1(H)\int_\RR \hat f(x)\overline{\hat g(x)}|x|^{1-2H}dx
	\end{equation}
	where
	\begin{equation*}
		c_1(H):=\lt(2\int_{\RR}\frac{1-\cos x}{|x|^{2H+1}}dx\rt)^{-1}= \frac {2 \pi}{\Gamma(2 H+1)\sin(\pi H)}\,.
	\end{equation*}
	As in \cite{MR1790083}, using \eqref{id:iso} and a denseness argument, the stochastic integration $\int_\RR f(s)dB_s^H$ can be extended to all integrands $f$ in $L^2(\RR)$ such that
	\begin{equation*}
		\int_\RR |\hat f(x)|^2|x|^{1-2H}dx<\infty\,.
	\end{equation*}
	Furthermore, if $f=1_{[a,b]}g$ with $g$ being a H\"older continuous function on $[a,b]$ of order $\gamma$, $\gamma+H>1$, then the integration $\int_\RR f(s)dB_s^H$ coincides with the Young integral (\cite{MR1555421}) $\int_a^b g(s)dB^H_s$.

	In what follows, we restrict the fractional Brownian motion $B^H$ on non-negative time domain and consider the following linear stochastic differential equation driven by $\{B^H_t,t\ge0\}$
	\begin{equation}\label{eqn:dxi}
		d \xi_t= \mu \xi_tdt+ \lambda dB^H_t\,,\quad t\ge0
	\end{equation}
	where $\mu\in\RR$ and $\lambda>0$. Given an initial datum, equation \eqref{eqn:dxi} has a unique solution
	\begin{equation}\label{def:xi}
		\xi_t=e^{\mu t}\xi_0+\lambda\int_0^t e^{\mu(t-s)}dB^H_s\ .
	\end{equation}
	Long term dynamics of \eqref{def:xi} exhibits three different behaviors depending on the sign of the parameter $\mu$. More precisely, if $\mu=0$, $\xi$ is nothing but a scalar multiple of the fractional Brownian motion $B^H$, whose variance grows to infinity in the long term. If $\mu<0$, $\xi$ is the so-called fractional Ornstein-Uhlenbeck process, whose variance converges to a finite limit. If $\mu>0$, the variance of $\xi$ grows exponentially in long term. 
	In this case, the process $\xi$ does not satisfy conditions \ref{con:s1s} nor \ref{con:ws1s}. Hence, in what follows, we mostly consider the cases when $\mu\le 0$.

	We show below that $\xi$ is indeed a Volterra-Gaussian process. Let us first define some notation. We denote
	\begin{equation*}
		c_2(H)=\lt(\frac{2H \Gamma(\frac32-H)}{\Gamma(H+\frac12)\Gamma(2-2H)}\rt)^{1/2}\,,
	\end{equation*}
	where $\Gamma$ is the Gamma function. 
	For each $H\in(0,\frac12)\cup(\frac12,1)$, let $K_H^\mu:\{(t,s)\in\RR^2:t>s\}\to\RR$ be a kernel defined by
		\begin{multline}\label{kernel1}
			K_H^\mu(t,s)=c_2(H)\lt[\lt(\frac ts\rt)^{H-1/2}(t-s)^{H-1/2}\rt.
			\\\lt.-s^{1/2-H}\int_s^t e^{\mu(t-u)}(H-\frac12-\mu u)u^{H-3/2} (u-s)^{H-1/2}  du \rt].
		\end{multline}
	For $H>\frac12$, we may integrate by parts
	\begin{align*}
		(H-\frac12) \int_s^t e^{\mu(t-u)}u^{H-3/2}(u-s)^{H-1/2}du
		&=(t-s)^{H-1/2}t^{H-1/2}
		\\&\quad-\int_s^t \frac{d}{du}\lt(e^{\mu(t-u)}(u-s)^{H-1/2} \rt)u^{H-1/2}du
	\end{align*} to obtain a simpler expression
		\begin{equation}\label{kernel2}
			K_H^\mu(t,s)=(H-\frac12) c_2(H) s^{1/2-H} \int_s^t e^{\mu(t-u)}u^{H-1/2}(u-s)^{H-3/2}du.
		\end{equation}
	For $H=\frac12$, we set $K_{1/2}^\mu(t,s) =e^{\mu(t-s)}$.
	The following scaling property of $K^\mu_H$ will be useful later. 
	\begin{lemma}
		For every $\kappa>0$ and $t>s>0$, we have
		\begin{equation}\label{id:Kscale}
			K^\mu_H(\kappa t,\kappa s)=\kappa^{H-\frac12} K^{\mu \kappa}_H(t,s)\,.
		\end{equation}
	\end{lemma}
	\begin{proof}
		This is merely a calculus exercise given \eqref{kernel1} and \eqref{kernel2}, we skip the details
	\end{proof}
	\begin{proposition}\label{prop:xiVol}
		$\xi$ is a canonical Volterra-Gaussian of the form \eqref{rep:volterra} with $K\equiv K^\mu_H$ and $U_tx=e^{\mu t}x$. More precisely, there exists a Brownian motion $\{W_t,t\ge0\}$ such that
		\begin{equation}\label{eqn:Volxi}
			\xi_t=e^{\mu t}\xi_0+\lambda\int_0^tK_H^\mu(t,u)dW_u\quad\forall t\ge0
		\end{equation}
		and $W$ and $\xi$ generate the same filtration $\cgg=\{\cgg_t,t\ge0\}$.
	\end{proposition}
	\begin{proof}
		The case $H=\frac12$ is trivial as $W=B^{\frac12}$. We consider $H\neq\frac12$ and denote $\gamma=H-\frac12$. It is shown in \cite{NorrosValkeilaVirtamo99} that there exists a Brownian motion $\{W_t,t\ge0\}$ which generates the same filtration as $B^H$ such that
		\begin{equation}\label{tmp:BHY}
			B^H_t=\int_0^t s^\gamma dY_s
		\end{equation}
		where $Y$ is the process defined by
		\begin{equation}\label{tmp:YW}
			Y_s=c_2(H)\int_0^s u^{-\gamma}(s-u)^{\gamma}dW_u\ .
		\end{equation}
		From \eqref{eqn:dxi} and \eqref{def:xi}, it follows that $\xi$ and $B^H$ generate an identical filtration, thus $\xi$ and $W$ also generate the same filtration. It suffices to find the form of kernel $K$. Using \eqref{tmp:BHY}, \eqref{tmp:YW} and integrating by parts (as is done in the proof of \cite[Theorem 5.2]{NorrosValkeilaVirtamo99}), we see that
		\begin{align*}
			&\int_0^t e^{\mu(t-s)}dB^H_s
			=\int_0^t e^{\mu(t-s)}s^\gamma dY_s
			=t^\gamma Y_t-\int_0^t \frac d{ds}[e^{\mu(t-s)}s^\gamma]Y_s ds
			\\&=c_2(H) t^{\gamma}\int_0^t u^{-\gamma}(t-u)^\gamma dW_u-c_2(H)\int_0^t e^{\mu(t-s)}(\gamma s^{\gamma-1}-\mu s^\gamma)\int_0^su^{-\gamma}(s-u)^\gamma dW_u ds \ .
		\end{align*}
		Upon changing the order of integrations, using stochastic Fubini theorem (cf. \cite[pg. 45]{BiaginiHuOksendalZhang08}), we arrive at
		\begin{equation*}
			\int_0^t e^{\mu(t-s)}dB^H_s=\int_0^t K_H^\mu(t,s)dW_s\ .
		\end{equation*}
		The assertion follows from here and \eqref{def:xi}. 
	\end{proof} 
	While working with a fixed memory $\xi[0,r]$, it is more convenient to interpret the It\^o integral $\int_0^rK^\mu_H(t,u)dW_u $ in a pathwise manner. The following result justifies this point. 
	\begin{lemma}\label{lem:youngK}
		Let $r$ be a fixed non-negative number. With $\PP$-probability one, 
		\begin{equation}\label{e4}
			\int_0^r K^\mu_H(t,s)dW_s=K^\mu_H(t,r)W_r-\int_0^r\partial_s K^\mu_H(t,s)W_sds
		\end{equation}
		for all $t>r$.
		The integral on the right-hand side above is an improper Riemann integral with singularity at $s=0$, that is
		\begin{equation*}
			\int_0^r\partial_u K^\mu_H(t,s)W_sds=\lim_{a\downarrow0}\int_a^r\partial_u K^\mu_H(t,s)W_sds
		\end{equation*}
	\end{lemma}
	\begin{proof}
		The case $H=\frac12$ is obvious. Herein, we assume $H\neq\frac12$. Fix $a\in(0,r)$, as is explained at the beginning of this section, the It\^o integral $\int_a^r K^\mu_H(t,s)dW_s$ coincides (almost surely) with its Young counter part. Hence, using integration by parts for Young integration (cf. \cite{HuLe}), 
		\begin{equation}\label{tmp:ar}
			\int_a^r K^\mu_H(t,s)dW_s=K^\mu_H(t,r)W_r-K^\mu_H(t,a)W_a-\int_a^r\partial_s K^\mu_H(t,s)W_sds\,.
		\end{equation}
		For the moment, we consider the case $H<\frac12$, we put $\gamma=\frac12-H$ and note that $\gamma\in(0,\frac12)$. We rewrite \eqref{kernel1} as
		\begin{equation*}
			K^\mu_H(t,s)=c_2(H)\lt[t^{-\gamma}s^\gamma(t-s)^{-\gamma}+\int_0^{t-s}e^{\mu(t-s-u)}(\gamma+\mu (u+s))s^\gamma(u+s)^{-\gamma-1}u^{-\gamma}du \rt]\,.
		\end{equation*}
		From here we obtain
		\begin{align}
			\frac{\partial_sK^\mu_H(t,s)}{c_2(H)}&=\gamma t^{ -\gamma+1}s^{-\gamma-1} (t-s)^{-\gamma-1} 
			-(\gamma+\mu t)s^\gamma t^{-\gamma-1}(t-s)^{-\gamma}
			\nonumber\\&+\mu s^{\gamma}\int_0^{t-s}e^{\mu(t-s-u)}(u+s)^{-\gamma-1}u^{-\gamma}du
			\nonumber\\&+s^{\gamma}\int_0^{t-s}e^{\mu(t-s-u)}(\gamma+\mu(u+s))(u+s)^{-\gamma-1}u^{-\gamma}\lt(-\mu+\frac \gamma s-\frac{\gamma+1}{u+s} \rt)du\,.
			\label{id:partial_sK}
		\end{align}
		To estimate $|\partial_s K^\mu_H(t,s)|$, we use the following two elementary estimates
		\begin{equation}
			 \int_0^{t-s}e^{\mu(t-s-u)}u^{-\gamma}(u+s)^{-\gamma}du\le	s^{- \gamma}\max\{1,e^{\mu t}\} \int_0^{t} u^{-\gamma}du
		\end{equation}
		and for $\theta>0$ such that $\gamma+\theta>1$,
		\begin{equation}
		 	\int_0^{t-s}e^{\mu(t-s-u)}u^{-\gamma}(u+s)^{-\theta}du\le s^{-\gamma- \theta+1}\max\{1,e^{\mu t}\} \int_0^{\infty}u^{-\gamma}(u+1)^{-\theta}du\,.
		\end{equation} 
		Bounding $t-s\ge t-r$, it follows that
		\begin{equation}\label{est:partialsK}
			|\partial_s K^\mu_H(t,s)|\lesssim 1+s^{-\gamma}+s^{-\gamma-1}\,,
		\end{equation}
		where the implied constant depends only on $\mu, \gamma,t,r$. 
		Since the sample paths of $W$ are almost surely $(1/2)^-$-H\"older continuous, $\lim_{s\downarrow0}s^{-\kappa}W_s=0$ for all $0<\kappa<1/2$. In particular, the Riemann integrals $\int_0^r s^{-\gamma}|W_s|ds$ and $\int_0^r s^{-\gamma-1}|W_s|ds$ are well-defined and finite. Together with \eqref{est:partialsK}, the integral
		\begin{equation*}
		 	\int_0^r\partial_u K^\mu_H(t,u)W_udu:=\lim_{a\downarrow0}\int_a^r\partial_u K^\mu_H(t,u)W_udu
		\end{equation*}
		is well-defined. The H\"older regularity of $W$ also yields 
		\begin{equation*}
			\lim_{a\downarrow0}K^\mu_H(t,a)W_a=0\,.
		\end{equation*}
		Sending $a$ to $0$ in \eqref{tmp:ar}, we obtain the result for $H<\frac12$. The case $H>\frac12$ is carried out analogously and easier, employing \eqref{kernel2} instead of \eqref{kernel1}, we skip the details.
	\end{proof}
	
	Let $r$ be a fixed non-negative number. As is described in Section \ref{sec:preliminaries}, one can construct a branching particle system $X=\{X_t,t\ge r\}$ starting from a memory $\xi[0,r]$ such that the spatial movement of each particle follows the law of $\xi$ conditioned on $\cgg_r$.
	The law of $X$ and its corresponding expectation are denoted by $\PP_{\xi[0,r]}$ and $\EE_{\xi[0,r]}$ respectively.

	Before stating limit theorems for this branching system, let us investigate the moments and regularity of $\xi$ conditional on $\cgg_r$. For conciseness (and without loss of generality), we identify $\PP_{\xi[0,r]}=\PP(\,\cdot\,|\cgg_r)$ and $\EE_{\xi[0,r]}=\EE(\,\cdot\,|\cgg_r)$.
\begin{lemma} For all $t\ge r$, we have
	\begin{equation}\label{eqn:m2xi}
		\EE_{\xi[0,r]} |\xi_t-e^{\mu t}\xi_0|^2\le\lt|\int_0^r K^\mu_H(t,u)dW_u\rt|^2+  c_1(H)\lambda^2  \int_\RR \frac{1-2e^{\mu t}\cos(xt)+e^{2 \mu t} }{\mu^2+ |x|^2}|x|^{1-2H}dx\,.
	\end{equation}
When $r=0$, \eqref{eqn:m2xi} becomes an equality. In particular,
	\begin{equation}\label{eqn:lim-}
		\lim_{t\to\infty}\EE|\xi_t|^2=c_1(H)\lambda^2 |\mu|^{-2H} \int_\RR \frac1{1+x^2}|x|^{1-2H}dx\quad\mbox{if}\quad \mu<0
	\end{equation}
		and
		\begin{equation}\label{eqn:lim+}
			\lim_{t\to\infty}e^{-2 \mu t} \EE|\xi_t-e^{\mu t}\xi_0|^2=c_1(H)\lambda^2|\mu|^{-2H}\int_\RR \frac1{1+x^2}|x|^{1-2H}dx\quad\mbox{if}\quad \mu>0\,.
		\end{equation}
	\end{lemma}
	\begin{proof}
		From \eqref{def:xi}, we see that
		\begin{equation}\label{tmp:xiK}
			\EE_{\xi[0,r]} |\xi_t-e^{\mu t}\xi_0|^2
			=\lambda^2\EE_{\xi[0,r]}\lt[\int_0^t e^{\mu (t-s)}dB_s^H\rt]^2 \,.
		\end{equation}
		On the other hand, from Proposition \ref{prop:xiVol}, 
		\begin{align*}
			\EE_{\xi[0,r]}\lt[\int_0^t e^{\mu(t-s)}dB_s^H \rt]^2
			&=\EE_{\xi[0,r]}\lt[\int_0^t K^\mu_H(t,s)dW_s \rt]^2
			\\&=\lt|\int_0^r K^\mu_H(t,s)dW_s\rt|^2+\int_r^t|K^\mu_H(t,s)|^2ds
			\\&\le \lt|\int_0^r K^\mu_H(t,s)dW_s\rt|^2+\int_0^t|K^\mu_H(t,s)|^2ds\,.
		\end{align*}
		Applying Proposition \ref{prop:xiVol} again and using \eqref{id:iso}, we see that
		\begin{align*}
			\int_0^t|K^\mu_H(t,s)|^2ds=\EE\lt[\int_0^t e^{\mu(t-s)}dB_s^H \rt]^2
			&=c_1(H) e^{2 \mu t} \int_\RR \lt|\frac{e^{-\mu t +\sqrt{-1} xt}-1}{- \mu+\sqrt{-1} x}\rt|^2|x|^{1-2H}dx
			\\&=c_1(H)\int_\RR \frac{1-2e^{\mu t}\cos(xt)+e^{2 \mu t} }{\mu^2+ |x|^2}|x|^{1-2H}dx\,.
		\end{align*}
		Hence, we obtain the following inequality
		\begin{multline}\label{e2}
			\EE_{\xi[0,r]}\lt[\int_0^t e^{\mu(t-s)}dB_s^H \rt]^2\le \lt|\int_0^r K^\mu_H(t,s)dW_s\rt|^2
			\\+ c_1(H)\int_\RR \frac{1-2e^{\mu t}\cos(xt)+e^{2 \mu t} }{\mu^2+ |x|^2}|x|^{1-2H}dx\,.
		\end{multline}
		When $r=0$, the above inequality becomes an equality. 	
		Combining \eqref{e2} with \eqref{tmp:xiK} yields \eqref{eqn:m2xi}. The two stated limits are simple consequences of \eqref{eqn:m2xi} and the following identity
		\begin{equation*}
			\int_\RR\frac{|x|^{1-2H}}{\mu^2+x^2}dx=\mu^{-2H}\int_\RR\frac{|x|^{1-2H}}{1+x^2}dx\,.
		\end{equation*}
		We conclude the proof.
	\end{proof}
	\begin{lemma}\label{lem:KtKs}
		For $\mu=0$, we have
		\begin{equation}\label{est:ErKts0}
			\lt(\EE_{\xi[0,r]} \lt|\int_r^t K^0_H(t,u)dW_u-\int_r^sK^0_H(s,u)dW_u \rt|^2\rt)^{\frac12}\le |t-s|^{H}\,.
		\end{equation}
		For each $\mu< 0$, there exists a positive constant $c(\mu,H)$ such that for every $t\ge s\ge r\ge0$, we have
		\begin{equation}\label{est:ErKts}
			\lt(\EE_{\xi[0,r]} \lt|\int_r^t K^\mu_H(t,u)dW_u-\int_r^sK^\mu_H(s,u)dW_u \rt|^2\rt)^{\frac12}\le  c(\mu,H)|t-s|+|t-s|^{H}\,.
		\end{equation}
	\end{lemma}
	\begin{proof}
		Since $W$ has independent increments, it follows that
		\begin{align*}
			&\EE_{\xi[0,r]} \lt|\int_r^t K^\mu_H(t,u)dW_u-\int_r^sK^\mu_H(s,u)dW_u \rt|^2
			\\&=\int_r^s |K^\mu_H(t,u)-K^\mu_H(s,u)|^2du +\int_s^t|K^\mu_H(t,u)|^2du 
			\\&\le \int_0^s |K^\mu_H(t,u)-K^\mu_H(s,u)|^2du +\int_s^t|K^\mu_H(t,u)|^2du 
			\\&=\EE \lt|\int_0^t K^\mu_H(t,u)dW_u-\int_0^sK^\mu_H(s,u)dW_u \rt|^2\,.
		\end{align*}
		Hence, it suffices to show \eqref{est:ErKts0} and \eqref{est:ErKts} for $r=0$. If $\mu=0$, in view of Proposition \ref{prop:xiVol}, the left-hand side of \eqref{est:ErKts0} with $r=0$ is the $L^2(\Omega)$-norm of $B_t^H-B^H_s$, which is exactly $|t-s|^H$. \eqref{est:ErKts0} is proved.
		Suppose now that $\mu<0$, in view of Proposition \ref{prop:xiVol}, \eqref{est:ErKts} with $r=0$ is equivalent to
		\begin{equation}\label{tmp:e3}
			\lt(\EE|\xi_t- \xi_s-(e^{\mu t}-e^{\mu s})\xi_0|^2\rt)^{\frac12}\le \lambda c(\mu,H)|t-s|+\lambda|t-s|^H\,.
		\end{equation}
		To show this, we start by writing \eqref{def:xi} into an integral form
		\begin{equation*}
			\xi_t-\xi_s =\mu\int_s^t \xi_u du+\lambda (B_t^H-B_s^H)\,.
		\end{equation*}
		In addition, we also have
		\begin{equation*}
			(e^{\mu t}-e^{\mu s}) =\mu\int_s^t e^{\mu u}du\,.
		\end{equation*}
		We denote by $\|\cdot \|_{2}$ the $L^2(\Omega)$-norm, using Minkowski inequality, we have
		\begin{equation*}
			\|\xi_t- \xi_s-(e^{\mu t}-e^{\mu s})\xi_0 \|_{2}\le |\mu|\int_s^t\|\xi_u-e^{\mu u}\xi_0\|_{2}du+\lambda\|B^H_t-B^H_s\|_{2} \,.
		\end{equation*}
		On the other hand, it follows  from \eqref{eqn:m2xi} that 
		\begin{equation*}
			\EE|\xi_u-e^{\mu u}\xi_0|^2\le \lambda^2c_1(H) \int_\RR \frac{4 }{\mu^2+ |x|^2}|x|^{1-2H}dx \,.
		\end{equation*}
		Upon combining the previous two estimates, we arrive at \eqref{tmp:e3}. The result follows.
	\end{proof}
	As an immediate consequence, we have the following tail estimate.
	\begin{proposition}\label{prop:estPmax}
		Suppose that $\mu\le 0$, $b>a>r\ge 0$ and $b-a\le 1$. For every $\epsilon>0$ and $p>\frac2H$, there exists a positive constant $C(\mu,\lambda,H,p)$ such that
		\begin{multline}
			\PP_{\xi[0,r]}(\sup_{s,t\in[a,b]}|\xi_t- \xi_s|\ge 3\epsilon )\le C(\mu,\lambda,H,p) \epsilon^{-p} |b-a|^{pH}
			\\ +1_{(|e^{\mu b}-e^{\mu a}||\xi_0|\ge \epsilon)}
			+1_{(\lambda\sup_{s,t\in[a,b]}|\int_0^r (K(t,u)-K(s,u))dW_u |\ge \epsilon )  }\,.
		\end{multline}
	\end{proposition}
	\begin{proof}
		We will use the following application of Garsia-Rodemich-Rumsey's inequality (cf. \cite{garsiarodemich}): for every $s,t\in[a,b]$ and continuous function $f$,
		\begin{equation*}
			|f_t-f_s|^p\le C^p |t-s|^{pH-2}\int_a^b\int_a^b\frac{|f_u-f_v|^p}{|u-v|^{pH}}dudv\,,
		\end{equation*}
		where $C$ is some absolute constant. (To obtain the above inequality, we choose $\Psi(u)=|u|^p $ and $p(u)=|u|^H$ in the notation of \cite{garsiarodemich}.) In this specific case, the above inequality is also called Morrey-Sobolev embedding inequality. In our situation, we put
		\[f_t:=\xi_t-e^{\mu t}\xi_0- \lambda\int_0^r K^\mu_H(t,u)dW_u=\lambda\int_r^tK^\mu_H(t,u)dW_u \,.\]
It follows that
	\begin{align*}
		&\PP_{\xi[0,r]} (\sup_{s,t\in[a,b]}|\xi_t- \xi_s|\ge 3\epsilon )
		\\&\le \PP_{\xi[0,r]} (\sup_{s,t\in[a,b]}|f_t- f_s|\ge \epsilon ) +\PP_{\xi[0,r]} (\sup_{s,t\in[a,b]}|e^{\mu t}-e^{\mu s}||\xi_0|\ge \epsilon )  
		\\&\quad+\PP_{\xi[0,r]} (\lambda\sup_{s,t\in[a,b]}|\int_0^r (K(t,u)-K(s,u))dW_u |\ge \epsilon )  
		\\&\le\PP_{\xi[0,r]} \lt( \int_a^b\int_a^b\frac{|f_u-f_v|^p}{|u-v|^{pH}}dudv\ge \frac{\epsilon^p}{C^p|b-a|^{pH-2}} \rt)  +1_{(|e^{\mu b}-e^{\mu a}||\xi_0|\ge \epsilon)}
		\\&\quad+1_{(\lambda\sup_{s,t\in[a,b]}|\int_0^r (K(t,u)-K(s,u))dW_u |\ge \epsilon )  }\,.
	\end{align*}
	The probability on the right-hand side can be estimated by Markov's inequality
		\begin{multline*}
			\PP_{\xi[0,r]}\lt( \int_a^b\int_a^b\frac{|f_u-f_v|^p}{|u-v|^{pH}}dudv\ge \frac{\epsilon^p}{C^p|b-a|^{pH-2}} \rt) 
			\\\le (\epsilon^{-1} C)^{p}|b-a|^{pH-2}\int_a^b\int_a^b\frac{\EE_{\xi[0,r]}|f_u-f_v|^p}{|u-v|^{pH}}dudv\,.
		\end{multline*}
		Note that for each $u,v$, $f_u-f_v$ is a centered Gaussian random variable, its $L^p(\Omega)$-norm is equivalent to its $L^2(\Omega)$-norm, which is estimated in Lemma \ref{lem:KtKs}. More precisely, we have
		\begin{equation*}
			\EE_{\xi[0,r]}|f_u-f_v|^p\lesssim (\EE_{\xi[0,r]}|f_u-f_v|^2)^{p/2}\lesssim |u-v|^{pH}
		\end{equation*}
		for all $u,v\in[a,b]$. The result follows upon combining these estimates together. 
	\end{proof}

\noindent\textbf{Branching fractional Brownian motion system.}
	Let us now consider the case $\mu=0$. In this case, $\xi=\lambda B^H$. In other words, the spatial motions follow the law of a fractional Brownian motion with intensity $\lambda$. Using the notation in Section \ref{sec:preliminaries}, we have $\sigma(t)=|\lambda|t^{H}$, which verifies \ref{c1} with $\ell=\infty$. To verify other hypothesis of Theorem \ref{thm:SLLN}, we first observe the following result.
	\begin{lemma}\label{lem:K0}
		If $H\in(\frac12,1)$, then
		\begin{equation}\label{limker1}
			\lim_{s\downarrow0}s^{H-\frac12}K_H^0(1,s)=\frac{c_2(H)}2\quad\mbox{and}\quad\lim_{t\to\infty}t^{1-2H} K^0_H(t,s)=\frac{c_2(H)}2 s^{\frac12-H}\,.
		\end{equation}
		If $H\in(0,\frac12)$, then
		\begin{equation}\label{limker2}
			\lim_{s\downarrow0}s^{\frac12-H}K_H^0(1,s)=\frac{H}{c_2(H)}\quad\mbox{and}\quad\lim_{t\to\infty}K^0_H(t,s)=\frac H{c_2(H)} s^{H-\frac12}\,.
		\end{equation}
	\end{lemma}
	\begin{proof}
		Suppose first $H>\frac12$. From \eqref{kernel2}, we have
		\begin{align*}
			\lim_{s\downarrow0} s^{H-\frac12}K_H^0(1,s)&=(H-\frac12)c_2(H)\lim_{s\downarrow0}\int_s^1 u^{H-1/2}(u-s)^{H-3/2}du
			\\&=(H-\frac12)c_2(H)\int_0^1 u^{2H-2}du\,,
		\end{align*}
which shows the first limit in \eqref{limker1}. The second limit in \eqref{limker1} is a consequence of the first and \eqref{id:Kscale}.
If $H<\frac12$, we use \eqref{kernel1}, then a change of variable to see that
\begin{align*}
\lim_{s\downarrow0} s^{\frac12-H}|K_H^0(1,s)|&=c_2(H) \lim_{s\downarrow0}
s^{1-2H}\lt[(1-s)^{H-\frac12}-\int_s^1 (H-\frac12)u^{H-\frac32}(u-s)^{H-\frac12}du
\rt]
\\&=(\frac12-H)c_2(H) \lim_{s\downarrow0}s^{1-2H}\int_s^1 u^{H-3/2}(u-s)^{H-1/2}du
\\&=(\frac12-H)c_2(H) \lim_{s\downarrow0}\int_1^{1/s} u^{H-3/2}(u-1)^{H-1/2}du
\\&=(\frac12-H)c_2(H) \int_1^{\infty} u^{H-3/2}(u-1)^{H-1/2}du\,.
\end{align*}
		Noting that the last integral above can be computed using the relation between Beta and Gamma functions,
		\begin{equation*}
		 	\int_1^\infty u^{H-3/2}(u-1)^{H-1/2}du=\frac{\Gamma(H+\frac12)\Gamma(1-2H)}{\Gamma(\frac32-H)}
		 	 =\frac{H}{(\frac12-H)c_2(H)^2}\,,
		\end{equation*} 
		we obtain the first limit in \eqref{limker2}. The second limit in \eqref{limker2} is a consequence of the first limit and \eqref{id:Kscale}.
	\end{proof}
	\begin{lemma}\label{lem:fBMtypmem}
		For fixed $r\ge0$, with probability one
		\begin{equation*}
			\lim_{t\to\infty}t^{-H}\int_0^r K^0_H(t,s)dW_s=0\,.
		\end{equation*}
		In other words, almost all sample paths of $\xi[0,r]$ are typical memories.
	\end{lemma}
	\begin{proof}
		In view of Lemma \ref{lem:youngK}, it suffices to show
		\begin{equation}\label{tmp:e6}
		 	\lim_{t\to\infty}t^{-H}K^0_H(t,r)=0
		\end{equation} 
		and
		\begin{equation}\label{tmp:e7}
			\lim_{t\to\infty}t^{-H}\int_0^r|\partial_s K^0_H(t,s)||W_s|ds=0\,.
		\end{equation}
		\eqref{tmp:e6} is a consequence of Lemma \ref{lem:K0}. We now consider \eqref{tmp:e7} in the case $H<\frac12$. As in the proof of Lemma \ref{lem:youngK}, putting $\gamma=\frac12-H$, we have
		\begin{align*}
			\frac{\partial_s K^0_H(t,s)}{c_2(H)} &=\gamma t^{-\gamma+1}s^{\gamma-1}(t-s)^{-\gamma-1}
			-\gamma s^\gamma t^{-\gamma-1}(t-s)^{-\gamma}
			\\&\quad+\int_0^{t-s} \gamma s^\gamma(u+s)^{-\gamma-1} u^{-\gamma}\lt(\frac \gamma s-\frac{\gamma+1}{u+s} \rt)du\,.
		\end{align*}
		Applying the following estimate
		\begin{equation*}
			\int_0^{t-s}(u+s)^{-\eta}u^{-\gamma}du 
			\le\int_0^{\infty}(u+s)^{-\eta}u^{-\gamma}du =s^{-\gamma- \eta+1}\int_0^\infty (u+1)^{-\eta}u^{-\gamma}du
		\end{equation*}
		which holds for $\eta>0$ such that $\gamma+\eta>1$, one has that 
		\begin{equation*}
			|\partial_s K^0_H(t,s)|\lesssim t^{-\gamma+1}(t-r)^{-\gamma-1}s^{-\gamma-1}+s^{\gamma}t^{-\gamma-1}(t-r)^{-\gamma}+ s^{-\gamma-1}\,,
		\end{equation*}
		where the implied constant is independent of $s$ and $t$. \eqref{tmp:e7} is deduced from here. In case $H>\frac12$, \eqref{tmp:e7} is proved analogously, we skip the details. 
	\end{proof}
	\begin{lemma}\label{lem:fbmregmem}
		For fixed $r\ge0$, with probability one the map $t\mapsto\int_0^r K^0_H(t,s)dW_s$ is uniformly continuous on $[r+1,\infty)$.
	\end{lemma}
	\begin{proof}
		In view of Lemma \ref{lem:youngK}, it suffices to show
		\begin{equation}\label{tmp:fbmdevK}
			\sup_{t\ge r+1}|\partial_t K^0_H(t,r)|\quad\mbox{and}\quad \sup_{t\ge r+1}\int_0^r|\partial_t\partial_s K^0_H(t,s)||W_s|ds
		\end{equation}
		are finite. We will derive estimates for partial derivatives of $K^0_H(t,s)$ below. The case $H=\frac12$ is trivial. Consider the case $H\neq\frac12$, from \eqref{kernel1} we obtain, putting $\gamma=H-\frac12$
		\begin{align*}
			\frac{\partial_t K^0_H(t,s)}{c_2(H)}= \gamma t^{\gamma}(t-s)^{\gamma-1}s^{-\gamma}
		\end{align*}
		and
		\begin{align*}
			\frac{\partial_t\partial_s K^0_H(t,s)}{c_2(H)}=\frac{\partial_s\partial_t K^0_H(t,s)}{c_2(H)}= \gamma t^{\gamma}(t-s)^{\gamma-2}s^{-\gamma-1}(s-\gamma t)\,.
		\end{align*}
		In particular, since $\gamma<\frac12$, the first supremum in \eqref{tmp:fbmdevK} is finite. In addition, for all $t>s$, we have
		\begin{equation*}
			|\partial_t\partial_s K^0_H(t,s)|\lesssim t^{\gamma+1}(t-r)^{\gamma-2}s^{-\gamma-1}\,.
		\end{equation*}
		Since $W$ is H\"older continuous of order $\theta$ for every $\theta<\frac12$, we have $|W_s|\lesssim s^\theta$ for all $s\in[0,r]$. It follows that
		\begin{equation*}
			\int_0^r|\partial_t\partial_s K^0_H(t,s)||W_s|ds\lesssim t^{\gamma+1}(t-r)^{\gamma-2}\int_0^r s^{\theta- \gamma-1} 	ds\,.
		\end{equation*}
		The integral on the right-hand side is finite as soon as we choose $\theta<\frac12$ so that $\theta>\gamma$, which is always possible because $\gamma<\frac12$.  This implies the second supremum in \eqref{tmp:fbmdevK} is also finite. 
	\end{proof}
	\begin{theorem}\label{thm:fbm}
		Let $\xi[0,r]$ be a typical memory of non-negative length $r$. Let $X=\{X_t,t\ge r\}$ be a branching particle system whose initial memory is $\xi[0,r]$ and underlying spatial movement is $\lambda B^H$. Assuming that \ref{c0} is satisfied and the map $t\to\int_0^r K^0_H(t,u)dW_u$ is uniformly continuous on $[ r+1,\infty)$, then with $\PP_{\xi[0,r]}$-probability one, for every continuous  function $f:\RR^d\to\RR$ such that $\sup_{x\in\RR^d}e^{\epsilon|x| }|f(x)|<\infty$ for some $\epsilon>0$, we have
		\begin{equation*}
			\lim_{t\to\infty}e^{-\beta t}t^{Hd}X_t(f)=e^{-\beta r} (2 \pi \lambda^2)^{-\frac d2}F\int_{\RR^d}f(y)dy \,,
		\end{equation*}
		where $F$ is the random variable defined in \eqref{def:F}.
	\end{theorem}
	\begin{proof}
		We simply verify the hypothesis of Theorem \ref{thm:SLLN}. \ref{c1} and \ref{con:U} are trivial with $\sigma(t)=\lambda t^H$, $\ell=\infty$, $U_tx=x$ and $U_\infty\equiv0$. Using the scaling relation \eqref{id:Kscale}, we see that
	\begin{align*}
		t^{-2H}\ln t\int_0^{\sqrt t} |K_H^0(t,s)|^2ds &=\ln t\int_0^{t^{-1/2}} |K_H^0(1,s)|^2ds
	\end{align*}
	which by Lemma \ref{lem:K0} and L'H\^opital's rule, converges to 0 as $t\to\infty$ (note that the case $H=\frac12$ is obvious). Thus \ref{con:s1s} is verified with $b(t)=\sqrt t$.

		For \ref{con:sl.tnmax}, we choose $t_n=r+n^{\kappa}$ where $\kappa$ is any fixed constant in $(0,1)$. It is easy to verify that this sequence satisfies \eqref{limtn} and \eqref{sumtn}. Let $\epsilon$ be a positive number. Let $n_0\in\NN$ be sufficiently large so that
		\begin{equation*}
			\lambda\sup_{s,t\in[r+n^\kappa,r+(n+1)^\kappa]} \lt|\int_0^r K(t,u)dW_u-\int_0^r K(s,u)dW_u\rt|< \epsilon
		\end{equation*}
		for all $n\ge n_0$. It is always possible to find such $n_0$ because $t\mapsto \int_0^rK(t,u)dW_u$ is uniformly continuous and $|(n+1)^{\kappa}-n^\kappa| \lesssim n^{-(1- \kappa)}$.  The estimate in Proposition \ref{prop:estPmax} yields
		\begin{align*}
			\PP_{\xi[0,r]}(\sup_{s,t\in[r+n^\kappa,r+(n+1)^\kappa]} |\xi_s- \xi_t|\ge 3\epsilon)
			\lesssim |(n+1)^{\kappa}-n^\kappa|^{pH}\lesssim n^{-pH(1- \kappa)}
		\end{align*}
		for every $n\ge n_0$ and $p>\frac2H$. Hence,
		\begin{align*}
			\sum_{n=n_0}^\infty n^{dH\kappa}\PP_{\xi[0,r]}(\sup_{s,t\in[r+n^\kappa,r+(n+1)^\kappa]} |\xi_s- \xi_t|\ge 3\epsilon)
			\lesssim \sum_{n=n_0}^\infty n^{dH\kappa-pH(1- \kappa) }\,,
		\end{align*}
		which is convergent when $p>\frac{dH \kappa+1}{H(1- \kappa)}$. Hence \ref{con:sl.tnmax} is verified. 
	\end{proof}
	\begin{remark}
		(i) Lemmas \ref{lem:youngK}, \ref{lem:fBMtypmem} and \ref{lem:fbmregmem} readily imply that almost all sample paths of $\xi[0,r]$ satisfy the hypothesis of Theorem \ref{thm:fbm}.

		\noindent(ii) It is interesting to observe that the limit object for the branching fractional Brownian system does not depend on the value of Hurst parameter $H$. Moreover, since $F$ is independent of the spatial motions, Theorem \ref{thm:fbm} indicates a universality phenomenon among the class of fractional Brownian motions with Hurst parameter $H$ varying in $(0,1)$.
	\end{remark}
	

\noindent \textbf{Branching fractional Ornstein-Uhlenbeck particle system.}
	Let us consider the case $\mu<0$. As noted earlier, in this case, the process $\xi$ is a fractional Ornstein-Uhlenbeck process. As in the case of fraction Brownian motions, we start with a few observations on the memories.
	\begin{lemma}\label{lem:fOUtypmem}
		For fixed $\mu<0$ and $r\ge0$, with probability one
		\begin{equation*}
			\lim_{t\to\infty}\int_0^r K_H^\mu(t,s)dW_s=0\,.
		\end{equation*}
		In other words, almost all sample paths of $\xi[0,r]$ are typical memories.
	\end{lemma}
	\begin{proof}
		In view of Lemma \ref{lem:youngK}, it suffices to show
		\begin{equation}\label{tmp:K}
			\lim_{t\to\infty}K^\mu_H(t,r)=0
		\end{equation}
		and
		\begin{equation}\label{tmp:sK}
			\lim_{t\to\infty}\int_0^r|\partial_s K^\mu_H(t,s)||W_s|ds=0\,.
		\end{equation}
We note that by L'H\^opital's rule, for every $\kappa\in\RR$ and $\eta>-1$,
		\begin{equation}\label{tmp:LH}
			\lim_{t\to\infty}t^{\kappa}(t-s)^{\eta} \int_s^t e^{\mu(t-u)}u^{-\kappa}(u-s)^{-\eta}du=-\mu^{-1}\,.
		\end{equation}
		The limit \eqref{tmp:K} is readily obtained by applying \eqref{tmp:LH} into \eqref{kernel1} and \eqref{kernel2}. 

		We now focus on showing \eqref{tmp:sK}. Using H\"older continuity of $W$ at $0$ (as in Lemma \ref{lem:fbmregmem}), it suffices to show
		\begin{equation}\label{tmp:633}
			\lim_{t\to\infty}\int_0^r|\partial_s K^\mu_H(t,s)|s^\theta ds=0\,,
		\end{equation}
		where $\theta$ is fixed in $(\gamma,1/2)$. We consider only the case $H<\frac12$ and leave the remaining case to the readers. 
		We recall that $\partial_s K^\mu_H(t,s)$ is computed explicitly in \eqref{id:partial_sK}. The integrations of the absolute values of the first two terms on the right-hand side of \eqref{id:partial_sK} with respect to the measure $s^\theta ds$ over $[0,r]$ obviously converge to 0 as $t\to\infty$. Concerning the remaining terms, their integration with respect to the measure $s^\theta ds$ over $[0,r]$ are constant multiples of the following integrals
		\begin{equation}\label{int1}
			\int_0^r s^{\gamma+\theta}\int_0^{t-s}(u+s)^{-\gamma-2}u^{-\gamma}du\,,\quad\int_0^r s^{\eta+\theta}\int_0^{t-s} e^{\mu(t-s-u)}(u+s)^{-\gamma-1}u^{-\gamma}duds
		\end{equation}
		and
		\begin{equation}\label{int2}
			 \int_0^r s^{\eta+\theta}\int_0^{t-s} e^{\mu(t-s-u)}(u+s)^{-\gamma}u^{-\gamma}duds\,,
		\end{equation}
		where $\eta$ ranges in $\{\gamma,\gamma-1\}$.
		The integrals in \eqref{int1} can be treated in the same way. Indeed, we note that for $\kappa\in\{\gamma+1,\gamma+2\} $
		\begin{equation*}
			\int_0^{t-s}e^{\mu(t-s-u)}(u+s)^{-\kappa}u^{-\gamma}du\le\int_0^\infty (u+s)^{-\kappa}u^{-\gamma}du=s^{1- \kappa- \gamma} \int_0^\infty (u+1)^{-\kappa}u^{-\gamma}du
		\end{equation*}
		and since $\theta>\gamma$,
		\begin{equation*}
			\int_0^r s^{\gamma+\theta}s^{1- \kappa- \gamma}ds<\infty\,.
		\end{equation*}
		We can apply \eqref{tmp:LH} and dominated convergence theorem to conclude that the integrals in \eqref{int1} converge to $0$ as $t\to\infty$. When $\kappa=\gamma$ the above argument breaks down because $\int_0^\infty (u+1)^{-\gamma}u^{-\gamma}du$ is infinite. We adopt a different strategy for the integrals in \eqref{int2}. Let $\rho$ be a fixed number in $(1- \gamma,1)$. Using the elementary estimate $e^{-x}\lesssim x^{-\rho}$ for all $x>0$, we get
		\begin{equation*}
			\int_0^{t-s}e^{\mu(t-s-u)}(u+s)^{-\gamma}u^{\gamma}du\lesssim s^{-\gamma}\int_0^{t-s}(t-s-u)^{-\rho}u^{-\gamma}du\lesssim s^{-\gamma}(t-s)^{1- \rho- \gamma}\,.
		\end{equation*}
		Hence, for $\eta\in\{\gamma,\gamma-1\}$ and $t>r$, we have 
		\begin{equation*}
			\int_0^r s^{\eta+\theta}\int_0^{t-s}e^{\mu(t-s-u)}(u+s)^{-\gamma}u^{-\gamma}duds
			\lesssim \int_0^r s^ {\eta+ \theta- \gamma} (t-s)^{1- \rho- \gamma}ds\,.
		\end{equation*}
		The integral on the right-hand side is at most $(t-s)^{1- \rho- \gamma	} \int_0^r s^ {\eta+ \theta- \gamma} ds$, which converges to $0$ as $t\to\infty$. Therefore, \eqref{tmp:633} is proved. 
	\end{proof}
	\begin{lemma}
		For fixed $\mu<0$ and $r\ge0$, with probability one, the map $t\mapsto\int_0^rK^\mu_H(t,s)dW_s$ is uniformly continuous on $[r+1,\infty)$.
	\end{lemma}
	\begin{proof} The proof resembles that of Lemma \ref{lem:fbmregmem}, involving computations of the partial derivatives of $K^\mu_H(t,s)$.  It is lengthy but straightforward. We skip the details.
	\end{proof}
	\begin{theorem}
		Let $\xi$ be the process defined in \eqref{def:xi} with $\mu<0$ and $\xi[0,r]$ be a typical memory of non-negative length $r$. Let $X=\{X_t,t\ge r\}$ be a branching particle system whose initial memory is $\xi[0,r]$ and the spatial movement of each particle follows the law of $\xi$ conditioned on $\cgg_r$. Assuming that \ref{c0} is satisfied and the map $t\to\int_0^r K^\mu_H(t,u)dW_u$ is uniformly continuous on $[ r+1,\infty)$, then with $\PP_{\xi[0,r]}$-probability one, for every bounded continuous functions $f$ on $\RR^d$,
		\begin{equation*}
			\lim_{t\to\infty}e^{-\beta t}X_t(f)=e^{-\beta r} (2 \pi \ell_H^2)^{-\frac d2}F\int_{\RR^d}e^{-\frac{|y|^2}{2\ell_H^2}} f(y)dy 
		\end{equation*}
		where $\ell_H$ is a finite positive constant given by
		\begin{equation}\label{def:ell}
			\ell_H^2=c_1(H)\lambda^2|\mu|^{-2H} \int_{\RR}\frac{|x|^{1-2H}}{1+x^2}dx\,.
		\end{equation}
	\end{theorem}
	\begin{proof}
		\ref{c1} and \ref{con:U} follow immediately from \eqref{eqn:Volxi} and \eqref{eqn:lim-}, with $\ell=\ell_H$, $U_t(x)=e^{\mu t}x$ and $U_\infty\equiv 0$. 

		We turn our attention to \ref{con:s1s} with $b(t)=t^{\delta}$ for $\delta\in(0,1)$ sufficiently small, which will be specified later. In what follows, we denote $\gamma=|H-\frac12|$ and note that $\gamma\in(0,\frac12)$. We need to show that
		\begin{equation*}
			\lim_{t\to\infty}\ln t \int_0^{t^\delta}|K^\mu_H(t,s)|^2ds =0\,.
		\end{equation*}
		If $H=\frac12$, $K^\mu_{1/2}(t,s)=e^{\mu(t-s)}$, it is obvious to verify the above limit. Consider the case $H>\frac12$, from \eqref{kernel2} we have
		\begin{equation*}
			|K^\mu_H(t,s)|\lesssim s^{-\gamma}t^{\gamma}\int_0^{t-s}e^{\mu(t-s-u)}u^{\gamma-1}du\,.
		\end{equation*}
		In addition, it follows from \eqref{tmp:LH} that
		\begin{equation*}
			\int_0^{t-s}e^{\mu(t-s-u)}u^{\gamma-1}du\lesssim (t-s)^{\gamma-1}\,.
		\end{equation*}
		Hence, 
		\begin{align*}
			\ln t\int_0^{t^\delta}|K^\mu_H(t,s)|^2ds\lesssim \ln t\int_0^{t^\delta}t^{2 \gamma}(t-s)^{2\gamma-2}s^{-2 \gamma}ds
			\lesssim (\ln t) t^{2 \gamma}(t-t^\delta)^{2 \gamma-2} t^{\delta(1-2 \gamma)}
		\end{align*}
		which converges to $0$ since $\delta<1$. Consider the case $H<\frac12$. In view of \eqref{kernel1}, it suffices to show
		\begin{equation*}
			\lim_{t\to\infty} \ln t\int_0^{t^{\delta}}|t^{-\gamma}(t-s)^{-\gamma}s^\gamma|^2ds 	=0
		\end{equation*}
		and
		\begin{equation}\label{tmp:22}
			\lim_{t\to\infty}\ln t\int_0^{t^\delta}\lt|s^\gamma \int_0^{t-s}e^{\mu(t-s-u)}(u+s)^{-\eta}u^{-\gamma}du \rt|^2ds 	=0
		\end{equation} 
		for $\eta\in\{\gamma,\gamma+1\}$.
		The first limit can be easily handled:
		\begin{align*}
			\ln t\int_0^{b(t)} |t^{-\gamma}(t-s)^{-\gamma}s^\gamma|^2ds
			\lesssim (\ln t) t^{-2 \gamma}(t-t^{\delta})^{-2\gamma}t^{\delta(1+2 \gamma)}
		\end{align*}
		which converges to $0$ as soon as we choose $\delta<\frac{4 \gamma}{1+2 \gamma}$. To show \eqref{tmp:22}, let us fix $\alpha$ in $(0,\gamma)$ if $\eta=\gamma$ and in $(\frac12,1- \gamma)$ if $\eta=\gamma+1$. From Jensen's inequality, we see that
		\begin{equation*}
			(u+s)^{-\eta}\lesssim u^{-\alpha} s^{\alpha-\eta}\,,
		\end{equation*}
		and hence, 
		\begin{equation*}
			\int_0^{t-s}e^{\mu(t-s-u)}(u+s)^{-\eta}u^{-\gamma}du\lesssim s^{\alpha- \eta}\int_0^{t-s}e^{\mu(t-s-u)}u^{-\alpha-\gamma}du\lesssim s^{\alpha- \eta}(t-s)^{-\alpha - \gamma}\,,
		\end{equation*}
		where we have used \eqref{tmp:LH} to bound
		\begin{equation*}
			\int_0^{t-s}e^{\mu(t-s-u)}u^{-\alpha-\gamma}du \lesssim (t-s)^{-\alpha-\gamma}\,.
		\end{equation*}
		It follows that
		\begin{align*}
			\int_0^{t^\delta}\lt|s^\gamma \int_0^{t-s}e^{\mu(t-s-u)}(u+s)^{-\eta}u^{-\gamma}du \rt|^2ds
			&\lesssim \int_0^{t^\delta}s^{2(\gamma+\alpha- \eta)}(t-s)^{-2(\alpha+\gamma)} ds
			\\&\lesssim (t-t^\delta)^{-2(\alpha+\gamma)}t^{\delta(2 \gamma+2 \alpha-2 \eta+1)}\,.
		\end{align*}
		We note that these estimates are valid because of the chosen range of $\alpha$. Hence, \eqref{tmp:22} is verified as soon as we choose $\delta<\frac{2 \alpha+2 \gamma}{2 \gamma+2 \alpha-2 \eta+1}$.

		Condition \ref{con:sl.tnmax} is satisfied with $t_n=r+ n^\kappa$ for any fixed $\kappa$ in $(0,1)$. The argument is similar to the proof of Theorem \ref{thm:fbm}. We omit the details.
\end{proof}

We conclude with a remark. One may wonder whether the limiting measure really depends on the initial position $\xi_0$. The answer is affirmative. Indeed, let $\xi_0\in\RR^d$ be fixed and consider a branching system such that the particle motions follow the law of the Gaussian process in \eqref{exp1}
and the branching factor $\beta>d$. It is easy to see that $\sigma^2(t)=(e^{2t}-1)/2 $, $\ell=\infty$ and $U_\infty=\lim_{t\to\infty}\frac{e^tId}{\sigma(t)}=Id$. The conditions \ref{c1} and \ref{con:U} are satisfied. To see that \ref{con:ws1s} holds, choose $b(t)=\kappa t$ for some constant $\kappa\in(\frac d \beta,1)$. Verifying the typical memory assumption is trivial. Hence, the conclusion of Theorem \ref{thm:wlln} is valid in this case and the right-hand side of \eqref{eqn:4.3} reads
\begin{equation*}
	e^{-\beta r} (2 \pi)^{-\frac d2}F e^{-\frac{|\xi_0|^2}2} \int_{\RR^d}f(y)dy\,.
\end{equation*}

\section{Acknowledgements}
Good refereeing is gratefully acknowledged by the authors.


\bibliography{../fbm}

\end{document}